\documentclass[10pt]{article}
\usepackage{hyperref}
\usepackage{amssymb,amsmath}
\usepackage{amsthm}
\usepackage{epsfig,graphicx}
\usepackage{xcolor}
\usepackage{amsbsy}

\newtheorem{theorem}{Theorem}
\newtheorem{lemma}{Lemma}
\newtheorem{proposition}{Proposition}

\newcommand{\eq}[1]{(\ref{#1})}

\newcommand{\la}{\lambda}
\newcommand{\CC}{\mathbb{C}}
\newcommand{\NN}{\mathbb{N}}
\newcommand{\RR}{\mathbb{R}}
\newcommand{\ZZ}{\mathbb{Z}}
\newcommand{\sign}{\mathop{\rm sign}\nolimits}
\newcommand{\dln}{\triangle n}

\begin{document}

\title{Direct Characterization of Spectral Stability \\ of Small Amplitude Periodic Waves in Scalar Hamiltonian Problems Via Dispersion Relation
\thanks{To appear in SIAM Journal on Mathematical Analysis. Published on arXiv with permission of The Society for Industrial and Apllied Mathematics (SIAM).}}

\author{Richard Koll{\'a}r \\
Department of Applied Mathematics and Statistics \\ 
Faculty of Mathematics, Physics and Informatics \\
Comenius University \\
Mlynsk{\'a} dolina,  842 48 Bratislava, Slovakia \\
e-mail: {\tt kollar@fmph.uniba.sk}\\
\\
Bernard Deconinck\\
Department of Applied Mathematics \\
University of Washington \\
 Seattle, WA, 98195, USA \\
e-mail: {\tt deconinc@uw.edu} \\
\\
Olga Trichtchenko \\
Department of Physics and Astronomy\\
University of Western Ontario\\
London, ON, N6A 3K7, Canada\\
e-mail: {\tt olga.trichtchenko@gmail.com}
}

\maketitle

\begin{abstract}
Various approaches to studying the stability of solutions of nonlinear PDEs lead to explicit formulae determining the stability or instability of the wave for a wide
range of  classes of equations. However, these are typically specialized to a particular equation and checking the stability conditions may not be not straightforward.  We
present results  for a large class of problems that reduce the determination of spectral stability of a wave to a simple task of locating zeros of  explicitly
constructed polynomials. We study  spectral stability of small-amplitude periodic waves in scalar Hamiltonian problems as a perturbation of the zero-amplitude
case. A necessary condition for stability of the wave is that the unperturbed spectrum is restricted to the imaginary axis. Instability can come about through a
Hamiltonian-Hopf bifurcation, i.e., of a collision of purely imaginary eigenvalues of the Floquet spectrum of opposite Krein signature. In recent work on the
stability of small-amplitude waves the dispersion relation of the unperturbed problem was shown to play a central role. We demonstrate that the dispersion
relation provides even more explicit information about wave stability: we construct a polynomial of half the degree of the dispersion relation, and its roots
directly characterize not only collisions of eigenvalues at zero-amplitude but also an agreement or a disagreement of their Krein signatures. Based on this
explicit information it is possible to detect instabilities of non-zero amplitude waves. In our analysis we stay away from the possible instabilities at
the origin of the spectral plane corresponding to modulation or Benjamin-Fair instability. Generalized KdV and its higher-order analogues are used as
illustrating examples.
\end{abstract}

\section{Introduction}
We study the spectral stability of small-amplitude periodic traveling waves in scalar Hamiltonian partial differential equations:
\begin{equation}
u_t = \partial_x  \frac{\delta H}{\delta u}\, .
\label{HPDE}
\end{equation}
Here
$$
u = u(x,t) = u(x+L,t), \quad x \in [0,L], \quad t > 0, \qquad \mbox{and}  \quad H = \int_0^L \mathcal{H} (u, u_x, \dots) \, dx
$$
is the Hamiltonian with density $\mathcal H$. Without loss of generality, we let $L = 2\pi$. This class of equations includes the Korteweg--de Vries (KdV), the generalized and modified KdV equations, the Kawahara equation, and other equations that arise in the study of dispersive problems in water waves, plasma physics etc. \cite{AS, IR}.

We assume that  \eq{HPDE} has a trivial solution, i.e.,
${\delta H}/{\delta u}  = 0$ for $u = 0$, and $H$ has an expansion $H = H^0 + H^1$, where $H^0$ is the quadratic part of $H$ and $H^1$ contains the higher order terms:
\begin{equation}
H^0 = -\frac{1}{2} \int_0^{2\pi} \sum_{j=0}^{N} \alpha_j (\partial_x^{j} u)^2\, .
\label{H0}
\end{equation}
As a consequence, all linear terms in (\ref{HPDE}) are of odd degree, as even degree terms would introduce dissipation. We assume that $N$ is a finite positive integer, and $\alpha_j \in \RR$. These assumptions exclude problems like the Whitham equation \cite{DT} ($N = \infty$) which remains a topic of investigation.

The now-standard approach to examine the stability of waves in Hamiltonian problems with symmetries is the theory developed by Vakhitov and Kolokolov \cite{VK} and Grillakis, Shatah, and Strauss \cite{GSS1, GSS2}, which allows for the determination of spectral stability of waves of arbitrary amplitude. In that setup, spectral stability implies orbital (nonlinear) stability under certain conditions, emphasizing the importance of the spectral information of the linearized problem.
Extensions of these results are found in \cite{KKS, KS2012, Pelinovsky2005, Pelinovsky2012}. Periodic problems within the same framework were considered in \cite{deckap, KapHar}. The use of any of these results relies on index theory
requiring additional information about the PDE. That information is typically provided, for instance, by assuming something about the dimension of the kernel of the linearized problem. For small-amplitude waves extra information is often obtained through a perturbation of the zero-amplitude problem. We avoid index theory and study directly the collision of eigenvalues. The parallel work \cite{TDK} illustrates how small-amplitude information is used to characterize the (in)stability of the waves. Here, we reduce the spectral stability problem for small-amplitude waves to the investigation of zeros of certain recurrently-defined polynomials, which appear in the theory of proper polynomial mappings \cite[p172]{DAngelo} and in orthogonal polynomial theory \cite[Chapter 18]{DLMF}.
To our knowledge, the connection between stability theory and these polynomials is new to the literature.

Our approach allows us to rigorously analyze the stability of KdV-type equations, including the
generalized KdV equation (gKdV), its higher order analogues, and also
the two-term balanced KdV equation. The results agree with the existing literature of spectral stability of periodic waves for gKdV and in the case of balanced high-order KdV equations they confirm and extend the analytical and numerical predictions in \cite{TDK}. Our method is closely related to the results in \cite{DT}, where the spectrum of small-amplitude periodic solutions of Hamiltonian PDEs is determined directly from the dispersion relation of the PDE linearized about the zero solution. Our theory adds to the results in \cite{DT}, and provides a simple and, importantly, a natural framework for studying the spectral stability of waves by perturbative methods. We refer the reader to \cite{DT} and \cite{TDK} for a number of  numerical illustrations of the results presented here for KdV-type equations.

The spectral stability of small-amplitude waves bifurcating from the trivial solution $u=0$ at a critical velocity $c=c_0$ can be examined using regular perturbation theory of the spectrum of \eq{HPDE} linearized about $u=0$ at $c  = c_0$. Our assumptions guarantee that $u=0$ is spectrally stable, i.e., the spectrum of the linearized problem is restricted to the imaginary axis, since \eq{HPDE} is Hamiltonian.

In the periodic setting the whole spectrum of the zero-amplitude problem is needed. However, Floquet theory \cite{KP} allows to decompose the continuous spectrum to an infinite union of sets of discrete eigenvalues of eigenvalue problems parametrized by the Floquet multiplier $\mu$. An important scenario for instability of small-amplitude waves on the bifurcation branch comes about through Hamiltonian-Hopf bifurcations \cite{MacKay,VDM1985} producing symmetric pairs of eigenvalues off the imaginary axis, i.e., exponentially growing and therefore unstable modes. Such bifurcations require non-simple eigenvalues of the linearized problem at zero amplitude, i.e., ``collided eigenvalues''. Furthermore, such colliding eigenvalues can split off from the imaginary axis only if they have opposite Krein signatures \cite{MacKay, KM2013}. Note that we stay away from the origin of the spectral plane and thus we do not consider modulation or Benjamin-Feir instability.

Both the location of the eigenvalues and their Krein signatures are characterized  by the dispersion relation of the linearized problem \cite{DT}.  We show that even the collision of eigenvalues and the agreement of their signatures is directly characterized by the dispersion relation. This characterization is through the roots of a polynomial, which is a reduction of the dispersion relation to a polynomial approximately half its degree. This is a surprising fact as it is by no means clear why such a characterization is possible, as the collisions of eigenvalues and their types are not itself objects that can be identified directly algebraically, particularly with a simpler algebraic relation than the eigenvalues themselves.

\section{General Setting}
\label{sec:GeneralSetting}

We follow the steps outlined in Section III of  \cite{DT}. We use a coordinate transformation $x \rightarrow x-ct$ to a frame moving with the wave,
\begin{equation}
\partial_t u = \partial_x \frac{\delta H}{\delta u} + c\partial_ x u = \partial_x \left( \frac{\delta H}{\delta u} + c u\right) =\partial_x \frac{\delta H_c}{\delta u},
\label{HPDEc}
\end{equation}
where $H_c$ is the modified Hamiltonian. The quadratic part of $H_c$ is
\begin{equation}
H_c^0 = \frac{c}{2} \int_0^{2\pi} u^2 \, dx - \frac{1}{2} \int_0^{2\pi} \sum_{j=0}^{N} \alpha_j (\partial_x^{j} u)^2 \, dx\, .
\label{Hc0}
\end{equation}
Traveling wave solutions of \eq{HPDE} are stationary solutions $U(x)$ of \eq{HPDEc} and stationary points of $H_c$.

\subsection{Perturbation from the trivial state. Dispersion relation}
For all $c\in \RR$, eq.~\eq{HPDEc} has the trivial solution $u(x,t)  = 0$.  We linearize \eq{HPDEc} about the zero solution to obtain an equation for the perturbation $v = v(x,t)$ from the trivial state
\begin{equation}
\partial_t v  = c\partial_x v - \sum_{j=0}^{N} (-1)^j \alpha_j \partial_x^{2j+1} v \, .
\label{lineqc}
\end{equation}
We decompose $v$ into a Fourier series in $x$, $v = \sum_{k=-\infty}^{\infty} \exp(ikx) \hat{v}_k$, to obtain decoupled evolution equations for each of the  Fourier coefficients $\hat{v}_k = \hat{v}_k(t)$:
\begin{equation}
\partial_t \hat{v}_k=-i \Omega(k)\hat{v}
_k
\qquad  k \in \ZZ,
\label{fourier}
\end{equation}
where $\Omega(k)$ is given by
\begin{equation}
\Omega(k) = \omega(k) - ck = \sum_{j=0}^N \eta_j k^{2j+1}\, , \qquad
\omega(k) = \sum_{j=0}^N \alpha_j k^{2j+1}\,, \qquad \eta_j = \alpha_j - c\delta_{j1}\, ,
\label{DR}
 \end{equation}
is the dispersion relation of \eq{lineqc}, obtained by letting $v(x,t) = \exp(ikx - i\Omega t)$ in \eq{lineqc}. Here $\omega = \omega(k)$ is the dispersion relation in the original frame of reference corresponding to \eq{HPDE}--\eq{H0}. Note that $\omega(k)$ is an odd function.

\subsection{Non-zero amplitude branches}
Next, we discuss non-zero amplitude periodic solution branches of \eq{HPDEc} bifurcating from the trivial state.
A requirement for this is that a non-trivial stationary solution of \eq{fourier} exists, i.e.,
$\Omega(k) = 0$, for $k \in \mathbb{N}$, since we have imposed that the solutions are $2\pi$ periodic. Thus
\begin{equation}
c = c_k = \frac{\omega(k)}{k}, \qquad k \in \NN.
\label{ckdef}
\end{equation}
For simplicity, we assume that a unique bifurcating branch emanates from $c=c_k$.  The solutions with $k > 1$ are $2\pi / k$ periodic. We focus on $k = 1$, i.e., $c = \omega(1)$. The cases with $k >1$ may be treated analogously (see Section~\ref{s:ktwo} for a discussion of the $k\ge 2$ in the context of gKdV equation).

\subsection{Floquet theory at zero amplitude}
Using Floquet theory \cite{DK2006, KP} the spectral stability of the non-trivial solution $U = U(x)$ of \eq{HPDEc} on the bifurcation branch starting at $c$  is determined by the growth rates of perturbations of the form
\begin{equation}
v(x,t) = e^{\lambda t}V(x),
~~
V(x) = e^{i\tilde{\mu} x}\sum_{n = -\infty}^{\infty} a_n e^{i nx}\, .
\label{Floquet}
\end{equation}

\noindent
Here $\tilde{\mu} \in (-1/2, 1/2]$ is the Floquet exponent.
Using \eq{fourier} for the zero-amplitude case,
\begin{equation}
\la  = \la_n^{(\tilde{\mu})}= -i \Omega(n+\tilde{\mu}) = - i \omega (n+\tilde{\mu}) + i (n+\tilde{\mu}) c, \qquad n \in \ZZ\, .
\label{spectrum}
\end{equation}
The expression \eq{spectrum} is an explicit expression for the spectrum of the linearized stability problem for solutions of zero amplitude. Next, we examine how the spectrum of the linearization changes as the solution bifurcates away from zero amplitude.

\subsection{Collisions of eigenvalues, Hamiltonian-Hopf bifurcations}
After Floquet decomposition \eq{Floquet}, the elements of the spectrum become eigenvalues of the $\tilde{\mu}$-parameterized operator obtained by replacing $\partial_x\rightarrow \partial_x+i \tilde{\mu}$ in the linear stability problem. The eigenfunctions associated with these eigenvalues are (quasi)periodic and are bounded on the whole real line, see \cite{KP,DO2011} for details.
For zero amplitude, the spectrum \eq{spectrum} is on the imaginary axis. Instabilities for small amplitude come about through collisions of purely imaginary eigenvalues at zero amplitude for a fixed value of $\tilde{\mu}$. Away from the origin, eigenvalues generically split off from the axis through the Hamiltonian-Hopf bifurcations \cite{MacKay,VDM1985} as the solution amplitude increases. Each such Hamiltonian-Hopf bifurcation produces a pair of eigenvalues off the imaginary axis that is symmetric with respect to the imaginary axis, thus yielding an exponentially growing eigenmode.

From \eq{spectrum}, it is easy to detect eigenvalue collisions away from the origin. They correspond to solutions of $\la_{n_1}^{(\tilde{\mu})} = \la_{n_2}^{(\tilde{\mu})} \neq 0$, $n_1,n_2 \in \ZZ$, $n_1 \neq n_2$, $\tilde{\mu} \in (-1/2, 1/2]$, i.e.,
\begin{equation}
-i\Omega(n_1+\tilde{\mu}) = -i\omega(n_1+\tilde{\mu}) + i (n_1+\tilde{\mu}) c = -i\omega(n_2+\tilde{\mu}) + i (n_2+\tilde{\mu}) c =- i\Omega (n_2 +\tilde{\mu})\, ,
\label{eq:collision}
\end{equation}
where $c=c_1$ is given by \eq{ckdef} with $k=1$. Solving this equation results in values of $\tilde{\mu}$ and $n_1$ for which $\lambda^{(\tilde{\mu})}_{n_1}$ is an eigenvalue colliding with another one. Typically this is done by solving \eq{eq:collision} for $\tilde{\mu}$ for different fixed $n_1$.

 \subsection{Krein signature}
A necessary condition for two eigenvalues colliding on the imaginary axis to cause a Hamiltonian-Hopf bifurcation is that the eigenvalues have opposite Krein signatures. The Krein signature is the sign of the energy of the eigenmode associated with the eigenvalue.  For a collision of eigenvalues to produce an instability this energy needs to be indefinite: a definite sign would entail bounded level sets of the energy, leading to perturbations remaining bounded.

For Hamiltonian systems with quadratic part given by \eq{Hc0} the eigenmode of the form
$v(x,t) = a_n \exp\left[i(n+\tilde{\mu}) x + \la_n^{(\tilde{\mu})}t\right] +  \mbox{c.c.}$, where c.c.~stands for complex conjugate of the preceding term, contributes to $H_c^0$ the relative energy (see \cite{DT})
$$
H_c^0|_{(n, \tilde{\mu})} \sim - |a_p|^2\, \frac{\Omega (n+\tilde{\mu})}{n+\tilde{\mu}}.
$$
Thus the Krein signature of $\la_n^{(\tilde{\mu})}$ is given by
\begin{equation*}
\kappa(\la_n^{(\tilde{\mu})}) = -\sign \left( \frac{\Omega(n+\tilde{\mu})}{n+\tilde{\mu}} \right)\, .
\label{Ksign}
\end{equation*}
A simple characterization of agreement of the signatures of two colliding eigenvalues $\la_{n_1}^{(\tilde{\mu})}$ and $\la_{n_2}^{(\tilde{\mu})}$  immediately follows.

\begin{proposition}
Let two eigenvalues $\la_{n_1}^{(\tilde{\mu})} = \la_{n_2}^{(\tilde{\mu})} = \la \neq 0$, $n_1 \neq n_2$,  of the Bloch wave decomposition \eq{Floquet} of
\eq{lineqc} coincide, i.e., \eq{eq:collision} holds. Then the product of Krein signatures of the eigenvalues is characterized by the sign of the quantity
\begin{equation}
q = q_{n_1,n_2}^{(\tilde{\mu})} = \frac{\Omega(n_1+\tilde{\mu})}{n_1+\tilde{\mu}} \cdot \frac{\Omega(n_2+\tilde{\mu})}{n_2+\tilde{\mu}}
= \frac{|\la|^2}{(n_1+\tilde{\mu})(n_2+\tilde{\mu})}\,
\label{qdef}
\end{equation}
Let $Z = Z_{n_1,n_2}^{(\tilde{\mu})} = (n_1+\tilde{\mu}) (n_2+\tilde{\mu})$. Since $\la \neq 0$ the sign of $Z$ characterizes an agreement of Krein signatures of the coinciding eigenvalues:
\begin{equation}
\kappa(\la_{n_1}^{(\tilde{\mu})}) \kappa(\la_{n_2}^{(\tilde{\mu})}) = \sign (q) = \sign \left[(n_1+\tilde{\mu})(n_2+\tilde{\mu})\right] = \sign (Z)\, .
\label{eq:Kreinproduct}
\end{equation}
\label{prop:1}
\end{proposition}

We denote
\begin{equation}
\mu := n_2 + \tilde{\mu}, \qquad \mbox{and} \qquad
\dln := n_1 - n_2\, .
\label{ndef}
\end{equation}
Here $\dln > 0$.   Then $Z = \mu (\dln+\mu)$ and the collision condition \eq{eq:collision} reduces to
\begin{equation}
\Omega(\dln+ \mu) = \Omega(\mu)
\, .
\label{resonance}
\end{equation}

\section{Recurrent Sequences of Polynomials}
\label{sec:RecurrentSequences}

Before we revisit \eq{resonance} in the next section, we need to define some particular recurrent sequences of polynomials.

\begin{lemma}
Let $a, b \in \CC$,  $m \in \mathbb{N}_0$, and
$$
t_m = a^m + (-b)^m\, .
$$
Then
\begin{equation*}
t_{m+1} = (a-b) t_m + (ab)t_{m-1}\, ,\qquad m \ge 1.
\label{recurtn}
\end{equation*}
\label{lemma:recurrent}
\end{lemma}
\vspace*{-0.3in}
\begin{proof}
$$
t_{m+1}  = (a-b)(a^m + (-1)^m  b^m) + ab (a^{m-1} + (-1)^{m-1} b^{m-1}) =
(a-b)t_m + (ab)t_{m-1}\, .
$$
\end{proof}

Since $t_0 = 2$ and $t_1 = a-b$, by induction all $t_m$ can be written as polynomials in the two variables $a-b$ and $ab$,
$t_m = t_m (a-b,ab)$.
Further, $t_m$ is a homogeneous polynomial in $a$ and $b$ of degree $m$. We introduce $s_m$ by $t_m  = (a-b)^m s_m (\gamma)$, i.e.,
\begin{equation}
s_m = s_m(\gamma) := \frac{t_m (a-b,ab)}{(a-b)^m}\, , \qquad
\mbox{with $\gamma := \displaystyle\frac{ab}{(a-b)^2}$.}
\label{def:ga}
\end{equation}
The sequence $s_m$ is characterized recursively by
\begin{eqnarray}
s_{m+1}  & = & s_m + \gamma s_{m-1}\, , \quad m \ge 1, \qquad s_0= 2, \quad s_1 = 1, \label{rec:s}
\end{eqnarray}
which shows that $s_m$ is a polynomial in $\gamma$ of degree $m/2$ ($m$ even) or $(m-1)/2$ ($m$ odd).
One can easily see that
\begin{gather*}
s_2(\gamma) =  1+ 2\gamma  , \quad
s_3(\gamma) =  1 + 3\gamma , \quad
s_4(\gamma) = 1 + 4\gamma + 2\gamma^2, \quad
s_5(\gamma) = 1 + 5\gamma + 5\gamma^2\, , \\
s_6(\gamma) = 1+ 6\gamma + 9 \gamma^2 + 3\gamma^3, \qquad
s_7(\gamma) = 1 + 7\gamma + 14\gamma^2 + 7 \gamma^3\, .
\end{gather*}
Solving the recurrence relationship,
\begin{equation}
s_m (\gamma) = \psi_+^m + \psi_-^m\, , \quad m \ge 0, \qquad
\psi_{\pm} := \frac{1}{2}\left(1 \pm \sqrt{1 + 4\gamma}\right)\, .
\label{s:exp}
\end{equation}
That  implies
\begin{equation}\label{lemma2}
s_m(0)=1,~~s_m(-1/4)=2^{1-m}.
\end{equation}
Note that $s_m(\gamma)$ is increasing on $(-1/4,0)$ as
\begin{equation}
s_m'(\gamma) = \frac{m}{\sqrt{1+4\gamma}} (\psi_+^{m-1} - \psi_-^{m-1}) > 0\, .
\label{growths}
\end{equation}
A few lemmas characterizing the behavior of $s_m(\gamma)$ are proved in the Appendix.

\section{Reduction of the Equation for Signatures of Colliding Eigenvalues}
\label{Sec:Reduction}

We prove that for scalar Hamiltonian problems \eq{HPDEc}--\eq{Hc0} of order $2N+1$, the polynomial equation \eq{resonance} characterizing the collision of eigenvalues with indices $n+\mu$ and $\mu$ at zero-amplitude resulting in Hamiltonian-Hopf bifurcations, and thus instability of non-zero amplitude periodic waves, can be expressed as a polynomial of  degree $N$ in a real variable $\gamma$ with coefficients independent of $\mu$, where $\gamma$ is defined as
\begin{equation}
\gamma :=\frac{ \mu (\triangle n + \mu)}{(\dln)^2}\, . \label{gammadef}
\end{equation}

\begin{theorem}
Let
$$
\Omega := \Omega(k)=\sum_{j=0}^N \eta_j k^{2j+1}\, ,
$$
be an odd polynomial of degree $2N+1$, $\eta_j \in \CC$ for $j=0,\dots, N$.
Then
\begin{equation}
\Omega(\dln+\mu) - \Omega(\mu) = \sum_{j=0}^{N} \eta_j ({\dln})^{2j+1}s_{2j+1}\left( \gamma \right)\, ,
\label{qeq}
\end{equation}
where the polynomial $s_{2j+1} = s_{2j+1}(\gamma)$ of degree $j$  is defined recurrently by \eq{rec:s}.
\label{theorem1}
\end{theorem}

\begin{proof}
The claim follows immediately by \eq{def:ga} and  Lemma~\ref{lemma:recurrent}  by setting $a := \dln + \mu$ and $b:=\mu$:
$$
\Omega(a) - \Omega(b)  = \sum_{j=0}^{N} \eta_j (a^{2j+1} - b^{2j+1})=  \sum_{j=0}^{N} \eta_j t_{2j+1}(a-b,ab)
=\sum_{j=0}^{N} \eta_j ({\dln})^{2j+1}s_{2j+1}(\gamma)\,.
$$
\end{proof}

As before, the collision condition \eq{resonance} expressed using \eq{qeq} is solved for $\gamma$ for different fixed values of $\dln$. After solving for $\gamma$, it is necessary to check that $\gamma$ gives rise to a real value of $\mu$ by solving the quadratic equation with the unknown  $\mu$:
$$
\mu(\mu + \dln) = \gamma (\dln)^2 \, .
$$
Thus
\begin{equation}
\mu_{1,2} = \frac{-\dln \pm \sqrt{(\dln)^2 + 4\gamma (\dln)^2 }}{2} = \frac{\dln}{2}\left(-1 \pm \sqrt{1 + 4\gamma} \right)
\, .
\label{mueq}
\end{equation}
By Proposition \eq{prop:1} we are interested in negative values of $\gamma$ that characterize a possible coincidence of two eigenvalues of opposite signature, as $\gamma$ has by \eq{gammadef} the same sign as $Z$ in \eq{eq:Kreinproduct}. Then any root $\gamma \in [-1/4,0)$ corresponds to a collision of two eigenvalues of opposite signature. If $\gamma < -1/4$, $\gamma$ does not correspond to a collision of two purely imaginary eigenvalues as $\mu$ is not real. If $\gamma >0$ then there is a collision of two eigenvalues of the same signature.
If $\gamma = 0$ the collision is located at the origin of the spectral plane, i.e., it does not correspond to the Hamiltonian-Hopf bifurcation.

We have proved the following main theorem characterizing the spectral stability of small-amplitude traveling waves of \eq{HPDE}.

\begin{theorem}
\sloppypar Consider a scalar $2\pi$-periodic Hamiltonian partial diff\-er\-en\-tial equation of the form \eq{HPDE} and assume that $u = 0$ is a spectrally stable solution.
Let \eq{DR} be the dispersion relation of the equation linearized about $u = 0$  in a reference frame moving with the velocity $c$. Then a branch of traveling wave solutions of \eq{HPDE} with velocity $c$ bifurcates from the trivial solution at $c = \omega(1)$, see \eq{ckdef}. A necessary condition for a Hamiltonian-Hopf bifurcation at zero-amplitude characterizing a loss of spectral stability of small-amplitude solutions on the bifurcating branch  is that
\begin{equation}
\sum_{j=0}^{N} \eta_j ({\dln})^{2j+1}s_{2j+1}(\gamma) = 0\,
\label{ceq}
\end{equation}
has  a root $\gamma$, $\gamma \in [-1/4, 0)$.
\label{th:gamma}
\end{theorem}

\section{Generalized KdV Equations}
As a simple example illustrating an application of Theorem~\ref{th:gamma} to study spectral stability of small-amplitude periodic traveling waves, we consider the generalized KdV equation  (gKdV)
\begin{equation}
\partial_t v + \alpha\partial_x^3 v + \partial_x f(v) = 0\, ,
\label{gKdV}
\end{equation}
and the generalized higher-order KdV equation ($p \ge 2$)
\begin{equation}
\partial_t v + \alpha\partial_x^{2p+1} v + \partial_x f(v) = 0\, .
\label{HOgKdV}
\end{equation}
Here we assume $f(0) = 0$ and periodic boundary conditions, $x\in [0,2\pi]$. Within this work we study high-frequency instabilities, staying away from the
origin in the spectral plane, i.e., we do not discuss the modulational or Benjamin-Feir instability. 

For simplicity we consider \eq{gKdV} first and then discuss the case of \eq{HOgKdV} as the reduction process and the results are completely analogous.
We will pay particular attention to the case of KdV equation with $f(v) = v^2$ in \eq{gKdV}.

For a detailed history of stability results of periodic traveling waves for KdV, mKdV (equation \eq{gKdV} with $f(u) = u^3$), and gKdV we refer the reader to \cite{BD2009, deckap},
see also \cite{KapHar, Johnson2009, DN2010}, see also \cite{DT}, Section 3.1, for numerical results illustrating the theory developed here. The results in the literature can be shortly summarized as:  periodic traveling waves are spectrally stable away from the origin
of the spectral plane (with the exception of cn solutions to mKdV), and also nonlinearly orbitally stable with respect to certain classes of perturbations. The
techniques used to prove the results for KdV are based on its integrability.

The dispersion relation of the linearization of \eq{gKdV} in the traveling frame is
\begin{equation}
\Omega = \Omega(k) = ck + \alpha k^3\, .
\label{OgKdV}
\end{equation}
Branches of small-amplitude waves are bifurcating from the trivial solution for the critical values of $c$ for which $\Omega(k) = 0$ for a nonzero integer value of $k$:
\begin{equation}
c_k = -\alpha k^2\, .
\label{cgKdV}
\end{equation}
Let us now fix $k\in \ZZ/\{0\}$ and set $c = c_k$. The condition for a collision of eigenvalues \eq{resonance} has the form
\begin{equation}
c\dln + \alpha \left[(\dln)^3 + 3\dln\mu (\dln+\mu)\right] = 0\, .
\label{Raux1}
\end{equation}
According to Theorem~\ref{th:gamma}  equation \eq{Raux1} can be rewritten in the form \eq{ceq}, i.e.,
\begin{equation}
c(\dln) + \alpha (\dln)^3 (1+3\gamma) = 0\, .
\label{redKdV}
\end{equation}
The root $\gamma$ of \eq{redKdV} that characterizes the nature of collisions of eigenvalues at zero amplitude is given by
\begin{equation}
\gamma =- \frac{c}{3\alpha (\dln)^2} - \frac{1}{3} = \frac{1}{3}\left( \frac{k^2}{(\dln)^2} - 1\right).
\label{gammaeq}
\end{equation}
The condition $-1/4 \le \gamma < 0$ can be expressed as
$$
-\frac{3}{4} \le \frac{k^2}{(\dln)^2} - 1 < 0, \qquad \mbox{i.e.,} \qquad
\frac{1}{4}(\dln)^2 \le k^2 < (\dln)^2\, ,
$$
or equivalently
\begin{equation}
k^2 < (\dln)^2 \le 4k^2\, , \qquad \mbox{and thus} \qquad
|k| < |\dln| \le 2|k|\, .
\label{kbound2}
\end{equation}
It is easy to see that in this special case the equality in the upper bound in \eq{kbound2} corresponds to a collision of eigenvalues $\lambda$ with indices $n_1+\tilde{\mu} = 1$ and $n_2+\tilde{\mu} = -1$ in \eq{spectrum}. But $\Omega(1) = \Omega(-1) = 0$ for (\ref{OgKdV}--\ref{cgKdV}). Thus the collision of opposite signature eigenvalues corresponding to the root $\gamma = -1/4$ in this particular case is located at the origin of the spectral plane and thus it is not associated with the Hamiltonian-Hopf bifurcation. Thus the instability condition is
\begin{equation}
|k| < |\dln| < 2|k|\, .
\label{kbound}
\end{equation}
Since the stability results are independent of $\alpha$ without loss of generality we assume $\alpha = 1$ in the rest of this section.

\subsection{gKdV Equation. Solutions with base period ${\pmb{2\pi}}$}
First, we consider KdV, i.e., $f(x) = u^2$, as the linear analysis is identical for all $f(x)$ satisfying $f(0) = 0$ and the characterization of the
collision condition in Theorem~\ref{th:gamma} does not dependent on the form of nonlinearity.
In that case, the solution branch indexed by $k=1$ bifurcating at $c_1 = -1$ from the trivial solution corresponds to the cnoidal waves with base-period $2\pi$, see \cite{DT}, Section 3.1, for the solution formula, numerical results, and analysis. The condition \eq{kbound} implies that collisions of eigenvalues with opposite Krein signature at zero-amplitude happen only for two eigenmodes of the form \eq{Floquet} with Fourier indices $n_1, n_2$, $\triangle n = n_1 - n_2$, where $1 < \triangle n < 2$.  As no such $\dln$ exists the small-amplitude cnoidal waves of base-period $2\pi$ are spectrally stable (away from the origin of the spectral plane). This is in agreement with the results obtained in \cite{BD2009} and \cite{DT}, Section 3.1, step 5.  The same result is true for any nonlinearity $f(x)$, including the case of mKdV, and thus, not accounting for a possible modulational instability, small-amplitude periodic traveling waves with base period $2\pi$ are spectrally stable for gKdV  \eq{gKdV}.

\subsection{KdV Equation. Solutions with base period $\pmb{2 \pi /k}$}
\label{s:ktwo}
We discuss the case $k\ge 2$. Solutions on the branch bifurcating  from the trivial solution at $c_k = -k^2$ also correspond in the case of KdV to the
cnoidal wave solutions, as the cnoidal waves comprise all periodic traveling waves to KdV. However, these solutions are subharmonic compared to the solutions on
the branch with index $k=1$, i.e., their base-period is $2\pi/k$. One way to see this is to consider \eq{gKdV} with $f(v) = v^2$ in the frame traveling with velocity $c$:
\begin{equation}
v_t +\alpha v_{xxx} + (v^2)_x + cv_x = 0\, .
\label{KawaharaEq}
\end{equation}
We set
\begin{equation}
y = \frac{x}{k}, \qquad \tau = \frac{t}{k^3}, \qquad u = k^2 v, \qquad \tilde{c} = k^2c.
\label{transform}
\end{equation}
Then \eq{KawaharaEq} transforms to
\begin{equation}
u_{\tau} + \alpha u_{yyy} + (u^2)_y + \tilde{c}u = 0\, .
\label{KawaharaEqTrans}
\end{equation}
Thus any solution $v(x,t)$ of \eq{KawaharaEq} with the base period $2\pi$ traveling with velocity $c$ corresponds 1-to-1 to a solution $u(y,\tau)$ of \eq{KawaharaEqTrans} with the base period $2\pi / k $ traveling with velocity $\tilde{c} = c k^2$. The $k$-repetition of $2\pi/k$-periodic solution of \eq{KawaharaEqTrans} is also a $2\pi$-periodic solution of \eq{KawaharaEqTrans} that is equivalent to \eq{KawaharaEq} with $c = c_k$. This relation allows to identify through \eq{transform}  the branch of $2\pi$ periodic solutions of \eq{KawaharaEq} bifurcating at $c=c_k$  with the branch of solutions of the same equation bifurcating at $c = c_1$, i.e., the branch of solutions of \eq{KawaharaEq} bifurcating at $c= c_k$ consists of properly rescaled multicopies of the solutions of the same equation located on the branch bifurcating at $c=c_1$. Therefore perturbations that are subharmonic for $k=1$ are co-periodic for $k\ge 2$, etc. This leads to more eigenvalue collisions for $k\ge 2$ than for $k=1$ since the co-periodic spectrum, e.g. the spectrum for $k\ge 2$ for the Floquet multiplier $\mu=0$  includes (after a proper rescaling) the union of the spectrum for $k=1$ and $\mu=0$, $\mu = 1/k$, $\mu = 2/k, \dots$.

\begin{figure}[h]
\centering
\includegraphics[width=\textwidth]{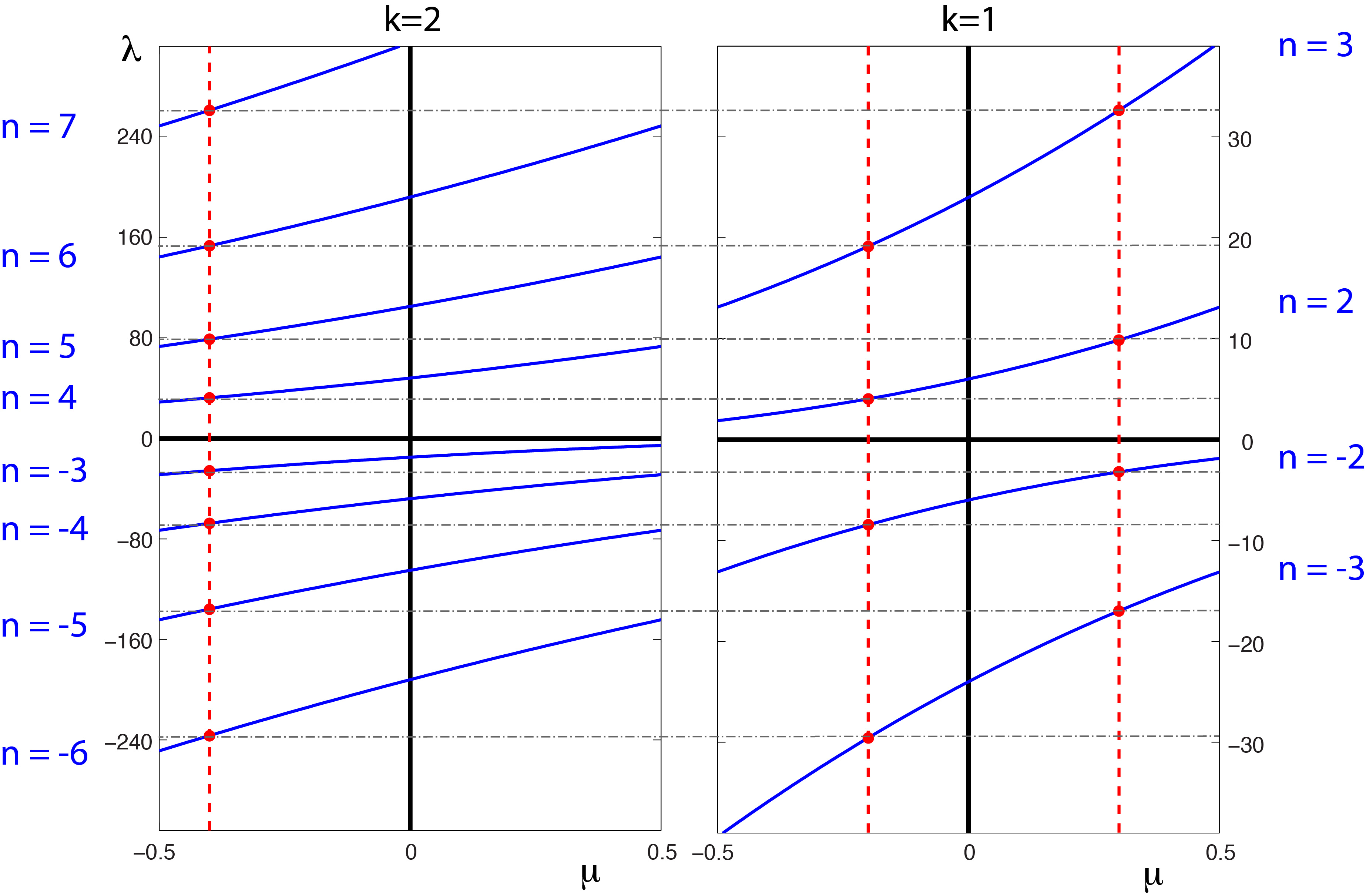}
\caption{
Illustration of the relation \eq{sigma8} of the spectrum $\sigma^{(2)}$ (left) and $\sigma^{(1)}$ (right) for KdV equation. Individual curves correspond to different values of $n$ with the index $n$ indicated. The spectrum partitions $\sigma_{\mu}$ correspond to all $\la$ for a given $\mu$.
Displayed are $\la = \la_n^{(\mu)}$ values for $\mu = -0.4$ ($k=2$, left) and  $\mu = -0.2$ and $\mu = 0.3$ ($k=1$, right).
For better visibility we have removed the branches with indices $n$, $-2\le n\le 3$ ($k=2$)  and $-1\le n \le 1$ ($k=1$), all undisplayed branches lie close to the horizontal axis. Note the scaling factor 8 on the $\la$ axis (left) for $\sigma^{(2)}$ compared to $\sigma^{(1)}$ (right). 
\label{Fig:k12}}
\end{figure}

As an illustration consider the case $k=2$. The spectrum of the linearized problem is given by
\begin{equation}
\sigma^{(2)} = \displaystyle\bigcup_{\mu \in (-1/2,1/2]} \sigma^{(k=2)}_{\mu} =
\left\{ \la_n^{(\mu)}; \ \la_n^{(\mu)} = - i \left[4(n+\mu) - (n+\mu)^3\right], n \in \ZZ \right\}\,.
\label{sigma2}
\end{equation}
On the other hand, the spectrum for $k=1$ is given by
\begin{equation}
\sigma^{(1)} = \displaystyle\bigcup_{\mu \in (-1/2,1/2]} \sigma^{(k=1)}_{\mu} =
\left\{ \la_n^{(\mu)}; \ \la_n^{(\mu)} = - i \left[(n+\mu) - (n+\mu)^3\right], n \in \ZZ \right\}\,.
\label{sigma1}
\end{equation}
It is easy to see (see Fig.~\ref{Fig:k12} for a visualization) that for all $\mu \in (-1/2,1/2]$
\begin{equation}
\frac{1}{8}\sigma_{\mu}^{(k=2)} = \sigma_{\mu/2}^{(k=1)} \cup \sigma_{\mu/2 + 1/2}^{(k=1)}\, .
\label{sigma8}
\end{equation}
Here multiplication of the set by a scalar means multiplication of each of its elements by the scalar and we use the periodicity $\sigma_{\mu} = \sigma_{\mu+1}$ for all $\mu \in \RR$ to properly define the second term $\sigma_{\mu/2 + 1/2}^{(k=1)}$.

The condition \eq{kbound} indicates that there are collisions of the eigenvalues of opposite signature at zero amplitude for modes of the form \eq{Floquet} for Fourier indices
$n_1, n_2$, with $\triangle n = n_1 - n_2$ satisfying $\triangle n \in \{ k+1, \dots, 2k-1\}$ and that is for $k\ge2$ a non-empty set. Generically, this would imply spectral instability of the waves. However, none of these collisions unfold for non-zero amplitude to a Hamiltonian-Hopf bifurcation.  Such bifurcations are not possible as according to
 \cite{BD2009} all periodic traveling wave solutions to KdV  are spectrally stable.  As a collision of eigenvalues of opposite Krein signature is only a necessary condition for a Hamiltonian-Hopf bifurcation, the analysis presented here does not allow to see  this phenomenon directly. Some indication can be found in the fact that these new collisions at $c = c_k$ correspond to collisions of opposite signature eigenvalues arising from different components (as opposed to from the same component)
of the union on the right hand side of \eq{sigma8}. The different spectrum partitions and associated eigenspaces do not interact with each other,  see   \cite{DeMeMa1992} and [Koll{\'a}r \& Miller, preprint 2018] for a throughout discussion of avoided Hamiltonian-Hopf bifurcations.

It is possible to see within the analysis presented here that the collisions of the opposite Krein signature eigenvalues of the $2\pi  / k$
periodic solutions are just an artifact of the $2\pi$ periodic setting, i.e., when one considers the stability of
the  $2\pi / k$ periodic solutions as the stability of its $k$-repetition in the $2\pi$ periodic frame in \eq{KawaharaEq}. Due to the
periodic character of the solution the stability of such a $k$-repetition is equivalent to the stability of a
single $2\pi / k$ periodic repetition in \eq{KawaharaEqTrans}. But we have proved above that the waves with period $L$ considered on the interval $[0,L]$ are spectrally stable (this corresponds to $k = 1$ for \eq{KawaharaEq} where we have set without loss of generality $L = 2\pi$). Therefore the $2\pi / k$ periodic waves are spectrally stable and
all collisions at zero amplitude of \eq{KawaharaEq} at $c = c_k$ are only due to multi-coverage of the spectrum $\sigma^{(k)}$ as in \eq{sigma8}.

The same argument can be used for gKdV with the nonlinearity $f(v) =v^n$, $n \ge 2$. However, in regard to the spectral stability of
small-amplitude waves lying on branches bifurcating at $c = c_k$ for $k \ge 2$ for a general $f(v)$, $f(0) = 0$, we can only conclude that there are collisions of the opposite signature eigenvalues at zero amplitude. A lack of a transformation analogous to \eq{transform}, that requires existence of a positive $r$ such that $f(au) = a^r f(u)$ for all $a \in \RR$, does not allow to rule out the potential Hamiltonian-Hopf bifurcations.

\subsection{Higher-order gKdV Equation}
A similar analysis can be performed for the higher-order gKdV equation \eq{HOgKdV}. In that case $\Omega(k) = -ck + (-1)^{p+1} \alpha  k^p$ and $c_k = (-1)^p \alpha k^{p-1}$. The relation $\Omega(n+\mu) = \Omega(\mu)$ reduces to a polynomial equation of degree $p$ for $\gamma$.
Similarly as for $p=1$ it is possible for $p=2$  to explicitly show that all the waves on the branch $k=1$ are spectrally stable, as none of the roots of $\Omega(\dln+\mu) = \Omega(\mu)$ in terms of $\gamma$ are located in the interval $(-1/4,0)$. To see this one needs to determine for which integer values of $\dln$ the roots of
$$
-k^4 + (\dln)^4\left( 1 + 5\gamma + 5\gamma^2\right) = 0
$$
lie in the interval $\gamma \in (-1/4,0)$. A short calculation reveals that the condition reduces to $|k| < |\dln| < 2|k|$, i.e. the same condition as for $p=1$ analyzed above leading to stability for $k=1$. The same statement can be proved for any $p\ge 1$ for which the equation for $\gamma$ has the form
\begin{equation}
-k^{2p} + (\dln)^{2p} s_{2p+1}(\gamma) = 0\, .
\label{gammap}
\end{equation}
There $s_{2p+1}(-1/4) = 2^{-2p}$ and  $s_{2p+1}(0) = 1$ by \eq{lemma2}, and also $s_{2p+1}(\gamma)$ is continuous on $[-1/4,0]$ and increasing on $(-1/4,0)$ by  \eq{growths}. Therefore the roots of \eq{gammap} lie in the interval $\gamma \in (-1/4,0)$ if and only if $|k| < |\dln| < 2|k|$. Hence the  small-amplitude periodic traveling wave solutions to \eq{HOgKdV} with the base period $2\pi$ ($k=1$) are spectrally stable, except perhaps with respect to modulational perturbations. The question of spectral stability of  small-amplitude wave solutions to \eq{HOgKdV} with the base period $2\pi/k$, $k\ge 2$ is not addressed here.

\section{Balanced Higher Order KdV equations}
\label{Sec:CaseStudy}
We demonstrate the full power of Theorem~\ref{th:gamma} on a more complicated example. Here we explicitly characterize stability regions for
small-amplitude periodic traveling wave solutions of KdV-type equations with two balanced linear terms of odd order:
\begin{equation}
u_t = \partial_x f(u)+ A\, \partial^{2q+1}_x  u + B \, \partial_x^{2p+1} u,
\label{KdVmn}
\end{equation}
subject to periodic boundary conditions.
Here $p > q$ are positive integers, $A, B \in \RR$ are non-zero coefficients, and $f(u)$ is a smooth function of $u$ and its spatial derivatives with $f(0)  = 0$, containing no linear terms.  The literature on this topic is limited. Most relevant is \cite{haraguslombardischeel}, where $f(u)\sim u^2$ (the Kawahara equation), and the period of the solutions is not fixed. It is concluded there that for solutions for which the amplitude scales as the 1.25-th power of the speed, solutions are spectrally stable. No conclusion is obtained for other solutions. Our investigation does not require this scaling, nor does it restrict the type of nonlinearity. Also relevant is \cite{deckap}, where the typical stability approach of \cite{KapHar} is extended to systems with singular Poisson operator like \eq{HPDE}, but the theory is not applied to \eq{KdVmn}. A mostly numerical investigation of equations like \eq{KdVmn} is undertaken in \cite{TDK}. As stated, our theory builds almost exclusively on \cite{DT} and our rigorous results agree with numerical results in \cite{TDK} where the special case $p = 2$, $q = 1$, and $A, B > 0$ was considered.

Traveling wave solutions $u=U(x-ct)$ with wave velocity $c$ satisfy
$$
-c U' =\partial _x f(U)+ A  U^{(2q+1)} + B U^{(2p+1)}.
$$
The spectral stability of small-amplitude waves that bifurcate at zero amplitude from the trivial solution $U=0$ is characterized by the growth of the solutions of  the linear equation
\begin{equation}
v_t = c v_x+A v_{(2q+1)x} + B v_{(2p+1)x},
\label{eqlinear}
\end{equation}
with dispersion relation
\begin{equation*}
\Omega = \Omega_{p,q}(k) = -ck - A (-1)^q k^{2q+1} - B (-1)^p k^{2p+1}=-ck-\alpha k^{2q+1}+\beta k^{2p+1},
\label{DRgeneral}
\end{equation*}
where we have introduced
\begin{equation}
\alpha = A (-1)^q, \qquad \qquad \beta = - B(-1)^p.
\label{ABdef}
\end{equation}
Without loss of generality, we assume that $\alpha > 0$. If not, the transformation $x \rightarrow - x$ (i.e., $k\rightarrow -k$), and $c \rightarrow -c$ can be used to switch the sign of $\alpha$. The scaling symmetry of the equation allows us to equate $\alpha=1$ hereafter. The choice of opposite signs in front of $\alpha$ and $\beta$ in \eq{ABdef} is intuitive: if $\alpha$ and $\beta$ have opposite sign the Hamiltonian energy \eq{Hc0} is definite and all eigenvalues have the same signature. This rules out Hamiltonian-Hopf bifurcations and the spectral instabilities following from them. In other words, the interesting case for our considerations is that both $\alpha$ and $\beta$ are positive. Lastly, since we study bifurcations from the first Fourier mode $k = 1$,
$c = \beta-\alpha=\beta-1$.

According to Theorem~\ref{theorem1}, eigenvalue collisions at zero-amplitude are characterized by the roots $\gamma$ of
\begin{equation*}
\dln R(\gamma) :=   -c\dln  - (\dln)^{2q+1} s_{2q+1}(\gamma) + \beta (\dln)^{2p+1} s_{2p+1}(\gamma) = 0.
\label{cd}
\end{equation*}
This is rewritten as
\begin{equation}
\beta \left[(\dln)^{2p} s_{2p+1}(\gamma)  - 1 \right]- \left[ (\dln)^{2q} s_{2q+1}(\gamma) - 1\right]= 0.
\label{eq10}
\end{equation}

Our goal is to find the parameter range $(\beta, \dln)$ for which the root $\gamma$ of \eq{eq10} satisfies $\gamma\in [-1/4,0)$.
The results obtained in the next section are graphically summarized in Fig.~\ref{Fig:region}.

\begin{figure}[h]
\centering
\includegraphics[width=0.85\textwidth]{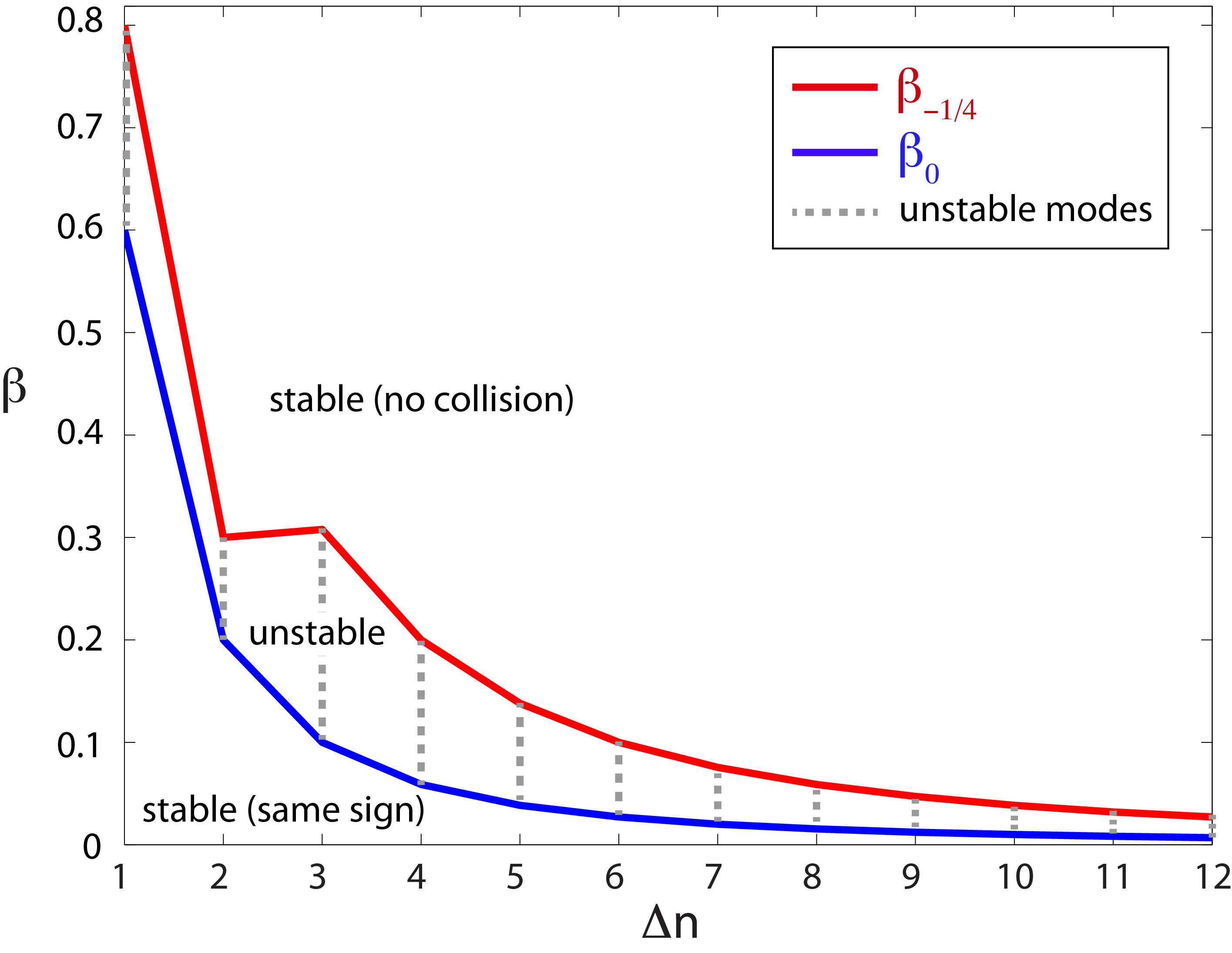}
\caption{Spectral stability regimes of the small-amplitude $2\pi$ periodic traveling waves for the Kawahara equation \eq{KdVmn}, $p=2$, $q=1$, $\alpha = 1$, $k=1$. Unstable pairs $(\dln, \beta)$ are indicated by the dashed line segments, stable pairs 
$(\dln, \beta)$ are above the curve $\beta= \beta_{-1/4}(\dln)$ and below the curve $\beta = \beta_0(\dln)$ given by \eq{beta00}--\eq{beta2} for $\dln \ge 3$, by \eq{betan2} for $\dln = 2$, and by $\eq{betan1}$ for $\dln = 1$. 
 \label{Fig:region}}
\end{figure}

An important role is played by the interval end points $\gamma = 0$ and $\gamma = -1/4$.
By \eqref{lemma2} for $\gamma= 0$ we have
$$
\beta((\dln)^{2p} - 1) - ((\dln)^{2q} - 1) = 0
$$
and therefore we set
\begin{equation}
\beta_0 = \beta_0(\dln) = \frac{(\dln)^{2q} -1}{(\dln)^{2p} - 1}.
\label{beta00}
\end{equation}
On the other hand \eq{eq10} reduces for $\gamma = -1/4$ by \eqref{lemma2} to
\begin{equation}
\beta_{-1/4}  = \beta_{-1/4}(\dln) = \left[\left(\displaystyle\frac{\dln}{2}\right)^{2q} - 1\right]
/ \left[\left(\displaystyle\frac{\dln}{2}\right)^{2p} - 1\right] \, .
\label{beta2}
\end{equation}

It follows immediately from Lemma~\ref{thetalemma} that for $\dln\geq 3$, $\beta_0(\dln)<\beta_{-1/4}(\dln)$, since this inequality may be rewritten as $f_{2p,2q}(\dln)<f_{2p,2q}(2)$ (in the notation of the Lemma).

\subsection{Collisions of eigenvalues of opposite signature}

Since the thresholds $\gamma = 0$ and  $\gamma = -1/4$ correspond, respectively, to $\beta = \beta_0(\dln)$ and  $\beta = \beta_{-1/4}(\dln)$,
where $\beta_0(\dln) < \beta_{-1/4}(\dln)$,  one may conjecture (for $\dln\geq 3$, since for $\dln=1, 2$ either $\beta_0$ or $\beta_{-1/4}$ is not defined) that collisions of eigenvalues of opposite Krein signature happen for $\beta \in (\beta_0(\dln), \beta_{-1/4}(\dln)]$.%
\footnote{Such a result would follow from monotonicity properties of the location of roots $\gamma$ with respect to $\beta$. Alternatively, we use an argument that proves that $\beta_0$ and $\beta_{-1/4}$ are the bounds of the stability region.}
For $\beta < \beta_0(\dln)$ one expects collisions of eigenvalues of the same signature and finally for $\beta > \beta_{-1/4}(\dln)$ one expects no collisions as the roots $\mu$ of \eq{mueq} are not real (see Fig.~\ref{Fig:beta}). As we prove next, this is true. The cases $\dln = 1$ and $\dln = 2$ are treated separately.

See \cite{TDK} for detailed numerical results (wave profiles and Fourier coefficients, spectrum diagrams)  in the case $p=2$, $q=1$ and $f(u) = u^2$ (Kawahara equation), particularly numerical simulations at non-zero amplitude confirming  presence of Hamiltonian-Hopf bifurcations (and thus spectral instability) that completely agree with the collisions of opposite Krein signature eigenvalues at zero-amplitude described here. In the numerical experiments all such collisions studied actually yielded the bifurcation.

\begin{figure}[h]
\centering
\includegraphics[width=0.7\textwidth]{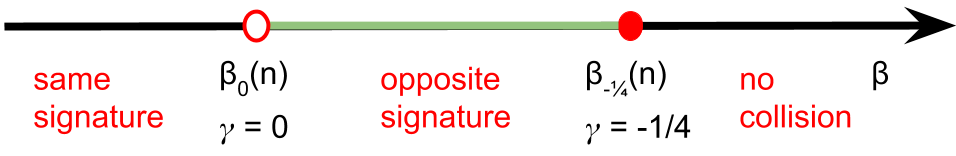}
\caption{Parameter regimes for $\beta$, $\beta \le \beta_0(\dln)$, $\beta \in (\beta_0(\dln), \beta_{-1/4}(\dln)]$, and $\beta > \beta_{-1/4}(\dln)$. \label{Fig:beta}}
\end{figure}

\begin{theorem}
{\bf Case $\mathbf{\dln \ge 3}$.}
Let $p, q$, $p > q$, be positive integers and  let $\dln$ is an integer, $\dln \ge 3$.  The presence and character of collisions of eigenvalues of the linearized problem \eq{eqlinear} at zero amplitude at $c = c_1=\beta-\alpha$
depends on the difference of the indices of the Fourier modes $\dln$ of the perturbation in the following way:
\begin{itemize}
\item[(i)]
If $\dln$ is such that $\beta < \beta_0(\dln)$,  then there is a collision of eigenvalues of the same signature, i.e., there is a root of \eq{eq10} with $\gamma > 0$ and there is no root with $\gamma \in [-1/4,0)$;
\item[(ii)]
If $\dln$ is such that $\beta_0(\dln) < \beta \le \beta_{-1/4}(\dln)$, then there is a collision of eigenvalues of opposite signature, i.e., there is a root $\gamma$ of \eq{eq10} such that $\gamma \in [-1/4, 0)$;
\item[(iii)]
If $\dln$ is such that $\beta_{-1/4}(\dln) < \beta$, then there is no collision of eigenvalues, i.e., all roots $\gamma$ of \eq{eq10} satisfy $\gamma <  -1/4$.
\end{itemize}
\label{maintheorem}
\end{theorem}

\begin{proof}

\noindent
{\bf Part (ii).} We show that for all $\dln \ge 3$ and $\beta_0(\dln) < \beta   \le \beta_{-1/4}(\dln)$
there exists $\gamma  \in [-1/4, 0)$  satisfying $R(\gamma) = 0$.
Therefore by \eq{eq10}, in such a parameter regime there is a collision of eigenvalues of opposite Krein signature.

It is easy to see that
$$
R(0) =   \beta[(\dln)^{2p}-1] - [(\dln)^{2q}-1]  > \beta_0 [(\dln)^{2p}-1] - [(\dln)^{2q}-1] = 0,
$$
and,
\begin{eqnarray*}
R( -1/4)& =&
\beta \left( \frac{(\dln)^{2p}}{2^{2p}} - 1\right) - \left(\frac{(\dln)^{2q}}{2^{2q}} - 1\right) \\
& \le &
\beta_{-1/4} \left( \frac{(\dln)^{2p}}{2^{2p}} - 1\right) - \left(\frac{(\dln)^{2q}}{2^{2q}} - 1\right)   = 0.
\end{eqnarray*}
Thus $R(0) > 0 \ge  R\left( -1/4 \right)$ and the polynomial $R(\gamma)$ has a real root $\gamma \in [-1/4, 0)$.

\vspace{\baselineskip}

\noindent
{\bf Part (i).}
Since $\beta < \beta_0(\dln) < \beta_{-1/4}(\dln)$ the same argument as in Part (ii) yields $R(-1/4) < 0$. Also,
$$
R(0) =  \beta[(\dln)^{2p}-1] - [(\dln)^{2q}-1]  < \beta_0 [(\dln)^{2p}-1] - [(\dln)^{2q}-1] = 0\, .
$$
We prove that $R(\gamma) = \beta [(\dln)^{2p} s_{2p+1}(\gamma) -1] - [(\dln)^{2q} s_{2q+1}(\gamma) - 1] < 0$ for all $\gamma \in [-1/4, 0]$.
By Lemma~\ref{slemma} for $\dln \ge 3$  and $p \ge 1$,
$$
(\dln)^{2p}s_{2p+1}(\gamma) \ge \frac{3^{2p}}{2^{2p+1}} > 1 \, .
$$
Thus for all $\gamma \in [-1/4,0)$   and $\beta < \beta_0$,
\begin{eqnarray}
R(\gamma)  & =&  \beta[(\dln)^{2p}s_{2p+1}(\gamma) - 1] - [\dln)^{2q}s_{2q+1}(\gamma) - 1] \nonumber \\
 & <&   \beta_0(\dln)  [(\dln)^{2p}s_{2p+1}(\gamma) - 1] - [(\dln)^{2q}s_{2q+1}(\gamma) - 1] \, .
\label{Req1}
\end{eqnarray}
We prove that the right-hand side of \eq{Req1} is non-positive. This is  equivalent to
\begin{equation}
\beta_0(\dln) = \frac{(\dln)^{2q}-1}{(\dln)^{2p}-1} \le \frac{(\dln)^{2q}s_{2q+1}(\gamma) - 1}{(\dln)^{2p}s_{2p+1}(\gamma) - 1}\, ,
\label{betaeq}
\end{equation}
or to
\begin{equation}
s_{2q+1} \ge s_{2p+1}[1-\theta(\dln)] + \theta(\dln)\, , \qquad \mbox{where} \quad
\theta(n) : =  \frac{(n)^{2p}-(n)^{2q}}{(n)^{2p+2q}-(n)^{2q}}.
\label{sest2}
\end{equation}
Clearly $0 < \theta(n) < 1$.
Since $s_{2p+1} < 1$ it suffices to prove \eq{sest2} for  $\dln$ that maximizes $\theta(\dln)$, $\dln\ge 2$.
However, by Lemma \ref{thetalemma} for $p > q \ge 1$,
$\max_{n \ge 2} \theta(n) = \theta(2)$ and it suffices to prove
$s_{2q+1} \ge s_{2p+1} (1-\theta (2)) + \theta(2)$,
i.e.,
\begin{equation*}
s_{2q+1} 2^{2q} (2^{2p} - 1) \ge s_{2p+1} 2^{2p} (2^{2q} -1) + 2^{2p} - 2^{2q}.
\label{sest3}
\end{equation*}
Therefore \eq{betaeq} follows directly from Lemma~\ref{decrease0} as it is equivalent for $p > q \ge 1$ to
\begin{equation*}
\frac{2^{2q} s_{2q+1} - 1}{2^{2q} - 1} \ge \frac{2^{2p}  s_{2p+1} - 1}{2^{2p} - 1} \, .
\label{separat}
\end{equation*}

Hence we proved $R(\gamma) < 0$ for all $\gamma \in [-1/4,0]$. On the other hand $R(\gamma)$ is an even order polynomial with a positive leading coefficient, i.e., $R(\gamma) \rightarrow \infty$ as $\gamma \rightarrow \infty$. Therefore there exists $\gamma_0 > 0$ such that $R(\gamma_0) = 0$. Such a root corresponds  by \eq{mueq} to a real value of $\mu$. Therefore in this regime there is a collision of two eigenvalues of the same signature.

\vspace{\baselineskip}

\noindent
{\bf Part (iii).}
Note that $R(0) >0$. We show that $R(\gamma) >0$ for $\gamma \ge -1/4$.
First,
$$
R( -1/4) =
\beta \left( \frac{n^{2p}}{2^{2p}} - 1\right) - \left(\frac{n^{2q}}{2^{2q}} - 1\right) >
\beta_{-1/4} \left( \frac{n^{2p}}{2^{2p}} - 1\right) - \left(\frac{n^{2q}}{2^{2q}} - 1\right)   = 0\, .
$$
For $\gamma \ge -1/4$,
\begin{eqnarray}
R(\gamma)  &= &
\beta \left[ (\dln)^{2p}s_{2p+1}(\gamma) - 1\right] - \left[ (\dln)^{2q} s_{2q+1}(\gamma) - 1\right] \nonumber \\
& > &
\beta_{-1/4}  \left[ (\dln)^{2p}s_{2p+1}(\gamma) - 1\right] - \left[ (\dln)^{2q} s_{2q+1}(\gamma) - 1\right]\, ,
\label{eq40}
\end{eqnarray}
since, by Lemma~\ref{slemma}, $(\dln)^{2p}s_{2p+1}(\gamma) \ge 1$.

We prove that
\begin{equation}
\frac{(\dln/2)^q-1}{(\dln/2)^p-1}
\ge
\frac{(\dln)^{q} s_{q+1}(\gamma) - 1}{ (\dln)^{p}s_{p+1}(\gamma) - 1}\, ,
\label{part3a}
\end{equation}
for any $p > q$. From \eq{eq40}, with $p \rightarrow 2p$ and $q \rightarrow 2q$, we obtain $R(\gamma) > 0$ for $\gamma \ge -1/4$.

Denote $m = \dln/2 \ge 1$ and $u_j = 2^j s_{j+1}$ for $j \ge 0$ to rewrite \eq{part3a} as
\begin{equation}
u_q \le u_p (1-\omega(m)) + \omega(m)\, , \qquad \mbox{where} \quad
\omega(m) = \frac{m^p-m^q}{m^{p+q}-m^q}.
\label{omegaeq}
\end{equation}
By Lemma~\ref{thetalemma}, the sequence $\omega(m)\in (0,1)$, is non-increasing for $m \ge 1$. Also, by Lemma~\ref{slemma},
$u_p = 2^p s_{p+1} \ge 1$, and \eq{omegaeq} follows from
$u_q \le u_p [1-\omega(1)] + \omega(1)$, where $\omega(1) = (p-q)/p$. Equation \eq{omegaeq} reduces to
$ (u_q - 1)/q \le (u_p-1)/p$, for $p > q \ge 1$. In terms of $s_q(\gamma)$ this is equivalent to
\begin{equation*}
\frac{2^q s_{q+1}(\gamma) - 1}{q} \le \frac{2^p s_{p+1}(\gamma) -1}{p}, \qquad \mbox{for $p > q \ge 1$},
\label{umono2}
\end{equation*}
which follows for $\gamma \ge -1/4$ from Lemma~\ref{decrease4}, since monotonicity of the positive sequence
$\displaystyle\frac{2^ms_{m +1}-1}{m(m+1)}$ directly implies monotonicity of the sequence $\displaystyle\frac{2^ms_{m +1}-1}{m}$.

Thus $R(\gamma) > 0$ for all $\gamma \ge -1/4$ and $R(\gamma)$ has no roots in $[-1/4, \infty)$ and there are no collisions of eigenvalues in this regime.
\end{proof}

For $\dln=1$, we use a similar argument. For $\dln= 1$ and $\gamma = 0$,
$R(0) = 0$. Hence $\gamma = 0$ is always a root of $R(\gamma) = 0$, corresponding to the relation\footnote{These eigenvalues are present due to symmetries; they do not leave the imaginary axis.} $
\Omega (1) = 0 = \Omega(0)
$.
For $p > q > 0$, denote
\begin{equation}
\beta_0^{(\dln=1)} = \frac{2q+1}{2p+1}\, ,
\qquad \mbox{and} \qquad
\beta_{-1/4}^{(\dln=1)} = \frac{1-2^{-2q}}{1 - 2^{-2p}}.
\label{betan1}
\end{equation}

\begin{theorem}
{\bf Case $\mathbf{\dln=1}$.}
Let $p, q$ be positive integers with $p>q$.
For the linearized problem \eq{eqlinear} at zero amplitude with $c = c_1$, the
presence and the character of eigenvalue collisions depend on the difference $\dln$ of the indices of the Fourier modes of the perturbation as follows:
\begin{itemize}
\item[(i)]
for $\beta < \beta^{(\dln=1)}_0$, eigenvalues of the same signature collide, i.e., there is a root of \eq{eq10} with $\gamma > 0$ and there is no root with $\gamma \in [-1/4,0)$;
\item[(ii)]
for $\beta^{(\dln=1)}_0 < \beta < \beta^{(\dln=1)}_{-1/4}$, eigenvalues of opposite signature collide, i.e., there is a root $\gamma$ of \eq{eq10} so that $\gamma \in [-1/4, 0)$;
\item[(iii)]
for $\beta_{-1/4}^{(\dln=1)} < \beta$, eigenvalues do not collide, i.e., $\gamma <  -1/4$, for all roots $\gamma$ of \eq{eq10}.
\end{itemize}
\label{maintheorem1}
\end{theorem}

\begin{proof}
First, we show that $\beta_0^{(\dln=1)}< \beta_{-1/4}^{(\dln=1)}$, which follows from the function $f(y)=(1-2^{-y})/(1+y)$ being decreasing for $y>2$. Its derivative has the numerator $(1+y)2^{-y}\ln 2+2^{-y}-1$, which is negative at $y=2$, and itself has a derivative that is negative for $y>2$.

Next, for $\beta \le \beta_{-1/4}^{(\dln=1)}$,
\begin{eqnarray}
R(-1/4)& = & \beta \left(s_{2p+1}(-1/4)-1\right)  -\left( s_{2q+1}(-1/4)-1\right)=
\beta(2^{-2p}-1) - (2^{-2q} - 1) \nonumber \\
& \ge & \beta_{-1/4}^{(n=1)} (2^{-2p}-1) - (2^{-2q} - 1) =  0\, ,
\label{eq90}
\end{eqnarray}
where equality holds only for $\beta = \beta_{-1/4}^{(\dln=1)}$.
On the other hand, if $\beta > \beta_{-1/4}^{(\dln=1)}$ then $R(-1/4) < 0$.
Further, for $\gamma = 0$ and all values of $\beta$, $R(0) = 0$. Finally, for $\gamma \in [-1/4,0)$
$$
R'(0) = \beta (2p+1)  - (2q+1).
$$
Therefore, for $\beta < \beta_0^{(\dln=1)}$,
\begin{equation}
R(0) = 0, \qquad R'(0) < 0,
\label{eq91}
\end{equation}
and, for $\beta > \beta_0^{(\dln=1)}$,
\begin{equation*}
R(0) = 0, \qquad R'(0) > 0.
\label{eq92}
\end{equation*}
Note that $R(0) = R'(0) = 0$ for $\beta = \beta^{(\dln=1)}_0$.

{\bf Part (i).}
By \eq{eq90} one has $R(-1/4) >0$, and by \eq{eq91} $R(0) = 0$ and $R'(0) < 0$. We prove that $R(\gamma) >0$ for all $\gamma \in [-1/4,0)$.
Thus $R = R(\gamma)$ does not have any roots in $(-1/4,0)$.
Moreover, $R(\gamma)$ is an odd-degree polynomial with a positive leading coefficient, $R(\gamma) \rightarrow \infty$ as $\gamma \rightarrow \infty$ and $R(0) =0$ and $R'(0) < 0$.
Therefore $R$ has a positive root.

Assume $\gamma \in [-1/4,0)$ and $\beta < \beta^{(\dln=1)}_0$. Then, using \eq{sest2a},
$$
R(\gamma) = \beta (s_{2p+1}(\gamma) -1) - (s_{2q+1}(\gamma)-1) > \beta^{(\dln=1)}_0  (s_{2p+1}(\gamma) -1) - (s_{2q+1}(\gamma)-1)\, .
$$
To establish $R(\gamma) > 0$ it is enough to prove
\begin{equation}
\beta^{(\dln=1)}_{0} \le \frac{s_{2q+1}(\gamma)-1}{s _{2p+1}(\gamma) - 1}, \qquad
\mbox{for $\gamma \in [-1/4,0)$.}
\label{g5}
\end{equation}
By Lemma~\ref{slemma} one has $s_m(\gamma) < 1$ for $m\ge 2$, $\gamma \in [-1/4,0)$. Hence \eq{g5} can be rewritten as
\begin{equation*}
\frac{s _{2p+1}(\gamma) - 1}{2p+1} \ge \frac{s_{2q+1}(\gamma)-1}{2q+1},
\label{g6}
\end{equation*}
which follows for $p > q > 0$ and $\gamma \in [-1/4, 0)$ from Lemma~\ref{decrease2}. Therefore $R(\gamma) > 0$ for $\gamma \in [-1/4,0)$.

\noindent
{\bf Part (ii).}
By \eq{eq90} one has $R(-1/4) >0$, and by \eq{eq91} $R(0) = 0$, $R'(0) > 0$. Therefore there exist a $\gamma \in (-1/4,0)$ such that $R(\gamma) = 0$.

\noindent
{\bf Part (iii).}
In this case  $R(-1/4) < 0$, and by \eq{eq91} $R(0) = 0$ and $R'(0) > 0$. We prove that $R(\gamma) < 0$ for $\gamma\in [-1/4,0)$ and $R(\gamma) > 0$ for $\gamma >0$. Therefore $R(\gamma)$ does not have a non-zero root for $\gamma \ge -1/4$.

First assume that $\gamma \in [-1/4,0)$. Then $\beta > \beta_{-1/4}^{(\dln=1)}$ implies, using \eq{sest2a},
$$
R(\gamma) =
\beta (s_{2p+1}(\gamma) - 1) - (s_{2q+1}(\gamma) - 1) < \beta_{-1/4}^{(\dln=1)} (s_{2p+1}(\gamma) - 1) - (s_{2q+1}(\gamma) - 1)\, .
$$
It suffices to prove
\begin{equation}
\beta^{(\dln=1)}_{-1/4} \ge \frac{s_{2q+1}(\gamma)-1}{s _{2p+1}(\gamma) - 1}, \qquad
\mbox{for $\gamma \in [-1/4,0)$,}
\label{g1}
\end{equation}
to establish $R(\gamma) < 0$.
The inequality \eq{g1} is rewritten as
\begin{equation*}
\frac{s _{2p+1}(\gamma) - 1}{2^{-2p}-1} \ge \frac{s_{2q+1}(\gamma)-1}{2^{-2q}-1},
\label{g2}
\end{equation*}
which follows from Lemma~\ref{increase2}. Thus $R(\gamma) < 0$ for $\gamma \in [-1/4,0)$.

Next, we assume $\gamma >0$. With $\beta > \beta_{-1/4}^{(\dln=1)}$ and using \eq{sest3a},
$$
R(\gamma) =
\beta (s_{2p+1}(\gamma) - 1) - (s_{2q+1}(\gamma) - 1) > \beta_{-1/4}^{(n=1)} (s_{2p+1}(\gamma) - 1) - (s_{2q+1}(\gamma) - 1)\, .
$$
It suffices to prove
\begin{equation*}
\frac{s _{2p+1}(\gamma) - 1}{2^{-2p}-1} \le \frac{s_{2q+1}(\gamma)-1}{2^{-2q}-1},
\label{g2h}
\end{equation*}
which follows from Lemma~\ref{increase2}. Thus $R(\gamma) > 0$ for $\gamma > 0$.
\end{proof}

It is easy to see that for $\dln= 2$, $R(-1/4)= 0$.
Thus $\gamma = -1/4$ is a root of $R(\gamma) = 0$ for all $\beta$.  It corresponds to the fact that
$\Omega(-1) = 0 = \Omega(1)$,  i.e., there is a collision of two eigenvalues of opposite Krein signature at the origin for all $\beta$.
This collision is due to the symmetries of the problem and these eigenvalues do not leave the imaginary axis in the weakly nonlinear regime. Thus this
collision does not affect stability. We focus on the remaining roots of $R(\gamma) = 0$.

We  denote
\begin{equation}
\beta_{0}^{(\dln=2)} =  \frac{2^{2q}-1}{2^{2p}-1}\, ,
\qquad \mbox{and} \qquad
\beta_{-1/4}^{(\dln=2)} = \frac{(2q+1)2q}{(2p+1)2p}\, .
\label{betan2}
\end{equation}
The inequality $\beta_0^{(\dln=2)}<\beta_{-1/4}^{(\dln=2)}$ follows similarly to $\beta_0^{(\dln=1)}<\beta_{-1/4}^{(\dln=1)}$, in the proof of the previous theorem.

\begin{theorem}
{\bf Case $\mathbf{\dln=2}$.}
Let $p, q$, $p > q$, be positive integers.
For the linearized problem \eq{eqlinear} at zero amplitude at $c = c_1$
the presence and the character of collisions of eigenvalues depends on the Fourier-mode parameter $n$ of the perturbation in the following way:
\begin{itemize}
\item[(i)]
for $\beta < \beta^{(\dln=2)}_0$, eigenvalues of the same signature collide, i.e. there is a root of \eq{eq10} with $\gamma > 0$ and there is no root with $\gamma \in (-1/4,0)$;
\item[(ii)]
for $\beta^{(\dln=2)}_0 < \beta < \beta^{(\dln=2)}_{-1/4}$, eigenvalues of the opposite signature collide, i.e. there is a root $\gamma$ of \eq{eq10} such that $\gamma \in (-1/4, 0)$;
\item[(iii)]
for $\beta_{-1/4}^{(\dln=2)} < \beta$, eigenvalues do not collide, i.e. all roots $\gamma$ of \eq{eq10} satisfy $\gamma \le  -1/4$.
\end{itemize}
\label{maintheorem2}
\end{theorem}

\begin{proof}
{\bf Part (i).}
We prove that $R(\gamma) < 0$, for $\gamma \in (-1/4,0)$. First, $R(\gamma)$ is an odd-degree polynomial and $R(\gamma) \rightarrow \infty$ as $\gamma \rightarrow \infty$ and $R(0) =0$ and $R'(0) < 0$.  Thus $R$ has a root $\gamma > 0$.

Assume $\gamma \in [-1/4,0)$ and $\beta < \beta^{(\dln=2)}_0$. Then
\begin{eqnarray*}
R(\gamma) &= & \beta (2^{2p}s_{2p+1}(\gamma) -1) - (2^{2q}s_{2q+1}(\gamma)-1)\\
& <&  \beta^{(\dln=2)}_0  (2^{2p}s_{2p+1}(\gamma) -1) - (2^{2q}s_{2q+1}(\gamma)-1)\, .
\end{eqnarray*}
To establish $R(\gamma) < 0$ it suffices to prove
\begin{equation*}
\beta^{(\dln=2)}_{0} \le \frac{2^{2q}s_{2q+1}(\gamma)-1}{2^{2p} s _{2p+1}(\gamma) - 1}, \qquad
\mbox{for $\gamma \in (-1/4,0]$.}
\label{g7}
\end{equation*}
This inequality is rewritten as
\begin{equation*}
\frac{2^{2p}s _{2p+1}(\gamma) - 1}{2^{2p}-1} \le \frac{2^{2q}s_{2q+1}(\gamma)-1}{2^{2q}-1},
\label{g8}
\end{equation*}
which follows from Lemma~\ref{decrease0}. Therefore $R(\gamma) < 0$ for $\gamma \in (-1/4,0]$.

\noindent
{\bf Part (ii).}
First,
\begin{eqnarray*}
R(0) &= & \beta (2^{2p} s_{2p+1}(0) - 1) - (2^{2q}s_{2q+1} - 1)
=  \beta (2^{2p}-1) - (2^{2q}-1) \\
& >& \beta_0^{(\dln=2)} (2^{2p} - 1) - (2^{2q} - 1) = 0.
\end{eqnarray*}
Next we show that $\lim_{\gamma \rightarrow -1/4^+} R'(\gamma) < 0$.
Indeed, for $\gamma > -1/4$, we have
\begin{eqnarray*}
R'(\gamma) &= & \beta \frac{2p+1}{\sqrt{1+4\gamma}} 2^{2p} (\psi_+^{2p} - \psi_-^{2p}) - 2^{2q}  (\psi_+^{2q} - \psi_-^{2q})\\
& <&  \beta_{-1/4}^{(n=2)} \frac{2p+1}{\sqrt{1+4\gamma}} \, 2^{2p} (\psi_+^{2p} - \psi_-^{2p}) - \frac{2q+1}{\sqrt{1+4\gamma}}  2^{2q}  (\psi_+^{2q} - \psi_-^{2q})
\end{eqnarray*}
as $\psi_+^2 > \psi_-^2 \ge 0$. The result follows from l'Hopital's rule, since
\begin{eqnarray*}
\lim_{\gamma \rightarrow -1/4^+}&&\!\!\!\!
\frac{(2q+1) 2^{2q} (\psi_+^{2q}(\gamma) - \psi_-^{2q}(\gamma))}{(2p+1)2^{2p} (\psi_+^{2p}(\gamma) - \psi_-^{2p}(\gamma))}\\
&&~~~~=\lim_{\gamma \rightarrow -1/4^+}
\frac{2q(2q+1) 2^{2q}\frac{1}{\sqrt{1+4\gamma}} (\psi_+^{2q-1}(\gamma) + \psi_-^{2q-1}(\gamma))}
{2p(2p+1) 2^{2p} \frac{1}{\sqrt{1+4\gamma}} (\psi_+^{2p-1}(\gamma) +  \psi_-^{2p-1}(\gamma))} \\
&&~~~~= \lim_{\gamma \rightarrow -1/4^+}
\frac{2q(2q+1) 2^{2q} s_{2q-1}(\gamma)}
{2p(2p+1) 2^{2p}  s_{2p-1}(\gamma)} \\
&&~~~~=  \frac{2q(2q+1) 2^{2q} 2^{-(2q-2)}}
{2p(2p+1) 2^{2p}  2^{-(2p-2)}} = \frac{2q(2q+1)}{2p(2p+1)}=\beta_{-1/4}^{(\dln=2)}.
\end{eqnarray*}

Thus $R(\gamma)< 0$ for $\gamma \in (-1/4,-1/4+\varepsilon)$, $\varepsilon > 0$, small. Since $R(0) > 0$, there exists $\gamma \in (-1/4,0)$ so that $R(\gamma) = 0$.

\noindent
{\bf Part (iii).}
We show that $R(\gamma) > 0$ for $\gamma > -1/4$. One has
\begin{eqnarray*}
R(\gamma) &= & \beta (2^{2p} s_{2p+1}(\gamma) - 1) - (2^{2q}s_{2q+1} (\gamma) - 1)\\
&>&  \beta_{-1/4}^{(\dln=2)} (2^{2p} s_{2p+1}(\gamma) - 1) - (2^{2q}s_{2q+1} (\gamma) - 1)\, .
\end{eqnarray*}
We show that
\begin{equation*}
\beta_{-1/4}^{(\dln=2)} = \frac{2q(2q+1)}{2p(2p+1)} \ge \frac{2^{2q}s_{2q+1} (\gamma) - 1}{2^{2p} s_{2p+1}(\gamma) - 1},
\label{h1}
\end{equation*}
which is equivalent to
$$
\frac{2^{2p}s_{2p+1}(\gamma) - 1}{2p(2p+1)} \ge \frac{2^{2q} s_{2q+1} (\gamma) - 1}{2q(2q+1)}.
$$
 This inequality follows from Lemma~\ref{decrease4}. Therefore $R(\gamma) = 0$ has no roots $\gamma > -1/4$ for $\beta > \beta_{-1/4}^{(\dln=2)}$.
\end{proof}

\section*{Appendix}

\begin{lemma}
Let $\alpha > 0$. The function
$$
g(x) = \frac{x\alpha^x}{\alpha^x - 1}
$$
is increasing on $(0,\infty)$.
\label{simple}
\end{lemma}

\begin{proof}
The condition $g'(x) > 0$ is equivalent to
$\alpha^x = e^{x \ln \alpha} > 1 + x \ln \alpha$.
This follows  directly from the Taylor expansion of $e^x$ at $x = 0$ with equality reached for $x = 0$.
\end{proof}

\begin{lemma}
Let $a > b > 0$. Define
$$
f (n) = f_{a,b}(n) = \frac{n^{a-b} - 1}{n^a-1}.
$$
We define $f(1) = \lim_{n \rightarrow 1} f(n) = (a-b)/a$.
Then $f(n)$ is a decreasing function on $[1,\infty)$.
\label{thetalemma}
\end{lemma}

\begin{proof}
The inequality $f'(n) < 0$  is equivalent to $a(n^b - 1) < b (n^a- 1)$, i.e.,
\begin{equation}
\frac{a}{b} < \frac{n^a-1}{n^b-1}\, .
\label{abeq}
\end{equation}
The estimate \eq{abeq} for $n > 1$  follows from the fact that the function
$$
h(n) = \frac{n^a - 1}{n^b-1}, \qquad a > b > 0,
$$
is increasing on $[1,\infty)$, where
$h(1) = \lim_{n\rightarrow 1} h(n) = a/b$.
The inequality $h'(n) > 0$ reduces to
$$
\frac{an^a}{n^a-1} > \frac{b n^b}{n^b - 1},
$$
which holds for $a > b > 0$ and $n > 1$ by Lemma~\ref{simple}.
Lemma~\ref{thetalemma} follows by continuity of $h(n)$ at $n=1$.
\end{proof}

\begin{lemma}
Let $s_m(\gamma)$ be as above. Then
\begin{eqnarray}
 s_m(\gamma) & \ge &  2^{-(m-1)}, \qquad  \mbox{for all $\gamma \ge -1/4$ and $m \ge 0$,}
\label{sest1} \\
s_m(\gamma) & < & 1, \qquad \quad \quad \  \ \, \mbox{for all $\gamma \in [-1/4,0)$ and $m \ge 2$,}
\label{sest2a} \\
s_m (\gamma) & > & 1, \qquad \quad \quad\  \ \, \mbox{for all $\gamma >0$ and $m \ge 2$.}
\label{sest3a}
\end{eqnarray}
\label{slemma}
\end{lemma}
\vspace*{-0.3in}
\begin{proof}
\sloppypar First, for $\gamma\geq -1/4$,
$s_m(\gamma)$ is an increasing function of $\gamma$ since
$s_m'(\gamma) = (m/\sqrt{1+4\gamma}) \left( \psi_+^{m-1}(\gamma) - \psi_-^{m-1} (\gamma)\right) > 0$.
The inequality \eq{sest1} follows from this and $s_m(-1/4)=2^{1-m}$.

Equation  \eq{sest2a} follows from the fact that $\psi_{\pm} \in (0,1)$ for $\gamma \in [-1/4, 0)$. Hence $s_{m+1}(\gamma) < s_m(\gamma)$ for all $m \ge 0$. Then $s_1(\gamma) = 1$ yields the claim.

Finally, we prove \eq{sest3a}. For $m=2$ and $m=3$, $s_2(\gamma) = 1 + 2\gamma > 1$, and $s_3(\gamma) = 1 + 3 \gamma > 1$ for $\gamma >0$. Then \eq{sest3a} follows directly from \eq{rec:s}.
\end{proof}

\begin{lemma}
For all $m \ge 0$ and $\gamma \ge -1/4$,
\begin{eqnarray}
s_{m+2}(\gamma) & \ge& -\gamma s_m(\gamma),
\label{db} \\
s_{m+1}(\gamma) &\ge& s_{m}(\gamma)/2,
\label{db1} \\
s_{m+1}(\gamma) &\le& \left[1+m(1+4\gamma)\right]s_m (\gamma)/2.
\label{TP1}
\end{eqnarray}
\label{doubling}
\end{lemma}

\vspace*{-0.3in}
\begin{proof}
The inequality  \eq{db} is equivalent to
$s_{m+2} - s_{m+1} + (s_{m+1} + \gamma s_m) \ge 0$. Using the recurrence relation \eq{rec:s}, it reduces to
$2s_{m+2} - s_{m+1}\ge 0$, i.e., $2s_{m+2} \ge  s_{m+1}$, $m \ge 0$.
Thus \eq{db} and \eq{db1} are equivalent except for \eq{db1} with $m = 0$, which is trivially satisfied ($2s_1 = 2 = s_0$).
Also note that $s_m(\gamma) \ge 0$ for $m \ge 0$ and $\gamma \ge 0$ and \eq{db} is satisfied for $\gamma \ge 0$. In the rest of the proof of \eq{db}, we assume that $m \ge 1$ and $\gamma \in [-1/4,0)$.
We shift $m \rightarrow m+1$ in \eq{db1}, $m \ge 0$, which becomes
\begin{equation}
\left(\psi_+ - \frac{1}{2}\right) \psi_+^{m+1} +\left(\psi_- - \frac{1}{2}\right)\psi_-^{m+1} \ge 0 \, .
\label{phi2}
\end{equation}
Since $\psi_- = 1-\psi_+$ for $\gamma \in [-1/4,0)$, \eq{phi2} is equivalent to
$$
\left(\psi_+ - \frac{1}{2}\right) \left[ \psi_+^{m+1} - \psi_-^{m+1} \right] \ge 0  \, ,
$$
which is satisfied for $\gamma \in [-1/4,0)$ since $\psi_+\geq 1/2$  and $\psi_+>\psi_-$. This proves \eq{db1} and \eq{db}.

We turn to \eq{TP1}.
Note that \eq{TP1} holds for $m=0$.
For $m \ge 1$, first we consider $\gamma \ge 0$. Using \eq{rec:s},
$$
2(s_m + \gamma s_{m-1}) \le \left[m(1+4\gamma) + 1\right] s_m,
$$
i.e.,
\begin{equation}
2\gamma s_{m-1} \le \left[m(1+4\gamma) - 1\right] s_m = (m-1) s_m + 4m\gamma s_m.
\label{TP2}
\end{equation}
But $m\ge 1$ and $s_m \ge 0$. Therefore $(m-1)s_m \ge 0$ and \eq{TP2} follows from $2\gamma s_{m-1} \le 4m \gamma s_m$, i.e., $s_{m} \ge s_{m-1}/2m$, which holds, according to \eq{db1}.

Next, consider $\gamma \in [-1/4,0)$. We write \eq{TP1} as
$2s_{m+1} - s_m \le m (1+4\gamma) s_m$,
and use \eq{s:exp} to obtain
\begin{equation*}
\psi_+^m \left(\psi_+ - \frac{1}{2}\right) + \psi_-^m \left( \psi_--\frac{1}{2}\right)  \le \frac{m (1+4\gamma)}{2} (\psi_+^m + \psi_-^m)\, .
\label{j1}
\end{equation*}
Using $\psi_+ + \psi_- = 1$,
$$
\left(\psi_+ - \frac{1}{2}\right) (\psi_+^m - \psi_-^m) \le  \frac{m (1+4\gamma)}{2} (\psi_+^m + \psi_-^m)\, .
$$
Since
$$
\psi_+ - \frac{1}{2} = \frac{\sqrt{1+4\gamma}}{2},
$$
Equation~\eq{TP1} is equivalent to
$$
(\psi_+^m - \psi_-^m) \le  m \sqrt{1+4\gamma}(\psi_+^m + \psi_-^m),
$$
or
\begin{equation}
\psi_+^m \left( 1 - m \sqrt{1+4\gamma}\right) \leq  \psi_-^m \left( 1 + m \sqrt{1+4\gamma}\right).
\label{j2}
\end{equation}
Both $\psi_+$ and $1 + m\sqrt{1+4\gamma}$ are positive, and
$$
\frac{\psi_-}{\psi_+} = \frac{1-\sqrt{1+4\gamma}}{1+\sqrt{1+4\gamma}} = \frac{1+2\gamma - \sqrt{1+4\gamma}}{-2\gamma}.
$$
It follows that proving
\eq{j2} is equivalent to proving
\begin{equation}
\frac{1-m\sqrt{1+4\gamma}}{1 +m\sqrt{1+4\gamma}} \le \left( \frac{1+2\gamma - \sqrt{1+4\gamma}}{-2\gamma}\right)^m \, .
\label{j3}
\end{equation}

We prove \eq{j3} by induction for $m \ge 0$. For $m = 0$, \eq{j3} is trivially satisfied. Assume that \eq{j3} holds for $m$. Using this, we have to show that \eq{j3} holds for $m+1$. This amounts to showing that
\begin{equation}
\frac{1-m\sqrt{1+4\gamma}}{1 +m\sqrt{1+4\gamma}} \, \frac{1+2\gamma - \sqrt{1+4\gamma}}{-2\gamma}\ge  \frac{1-(m+1)\sqrt{1+4\gamma}}
{1 +(m+1)\sqrt{1+4\gamma}}.
\label{j4}
\end{equation}
Multiplying \eq{j4} by all (positive) denominators simplifies to an inequality which holds for all $\gamma \in [-1/4,0)$:
\begin{equation*}
m(m+1)(1+4\gamma)^{3/2}\left( 1 - \sqrt{1+4\gamma}\right) \ge 0.
\end{equation*}
\end{proof}

\begin{lemma}
For all $m\geq 2$,
\begin{eqnarray}
-\gamma (2^m - 1)s_{m-1}(\gamma)  + s_{m+1} (\gamma) &\ge& 1\, ,
\
\mbox{for $\gamma\in [-1/4,0]$.}
\label{gamest} \\
-\gamma (2^m - 1)s_{m-1}(\gamma)  + s_{m+1} (\gamma) &\le& 1\, ,
\
\mbox{for $\gamma \ge 0$.}
\label{gamest2}
\end{eqnarray}
\label{gamlemma}
\end{lemma}
\begin{proof}
We prove\eq{gamest} using induction. For $m=2$ and $m=3$
\begin{eqnarray*}
-\gamma (2^2 - 1) s_1(\gamma) + s_3(\gamma)  =  -3\gamma + 1 + 3\gamma & =&  1, \\
-\gamma(2^3 - 1)s_2(\gamma) + s_4(\gamma)   =  1 - 3\gamma(1+4\gamma) & \ge & 1.
\end{eqnarray*}

Assume \eq{gamest} holds for some $m\ge 3$, i.e.,
\begin{equation}
-\gamma(2^m -1)s_{m-1} + s_{m+1} \ge 1.
\label{ind1}
\end{equation}
By Lemma~\ref{doubling}, $s_{m} + \gamma s_{m-2} \ge 0$.  Using \eq{rec:s} this becomes $s_{m-1} + 2 \gamma s_{m-2} \ge 0$.
After multiplication by $2^m - 1 >0$, we obtain the equivalent form
$$
(2^m - 1) s_{m-1} + 2 \gamma (2^{m}-1) s_{m-2}  =
(2^m - 1) s_{m-1} + \gamma (2^{m+1}-2) s_{m-2}  \ge 0,
$$
which, using \eq{rec:s}, is rewritten as
\begin{equation}
2^{m} s_{m-1} + \gamma (2^{m+1} - 1)  s_{m-2} - s_{m} \ge 0.
\label{ind2}
\end{equation}
Multiplying \eq{ind2} by $-\gamma \ge 0$ and adding \eq{ind1} gives
$$
-\gamma(2^{m+1} -1) (s_{m-1} + \gamma s_{m-2}) + ( s_{m+1}  + \gamma s_m) \ge 1,
$$
which is rewritten as
$$
-\gamma(2^{m+1} -1) s_m +s_{m+2}\ge 1\, .
$$
This concludes the proof of the second induction step.

Next we prove \eq{gamest2}. The statement is true for $m = 2$ and $m=3$:
$$
-\gamma (2^2 - 1) s_1(\gamma) + s_3(\gamma)  = 1\, , \qquad \quad
-\gamma(2^3 - 1)s_2(\gamma) + s_4(\gamma) = 1 - 3\gamma - 12 \gamma^2 \le 1.
$$
Assume \eq{gamest2} holds for some $m\ge 3$, i.e.,
\begin{equation}
-\gamma(2^m -1)s_{m-1} + s_{m+1} \le 1 .
\label{ind10}
\end{equation}
By Lemma~\ref{doubling}, $s_{m} + \gamma s_{m-2} \ge 0$ or equivalently $s_{m-1} + 2 \gamma s_{m-2} \ge 0$, so that
$$
(2^m - 1) s_{m-1} + 2 \gamma (2^{m}-1) s_{m-2}  =
(2^m - 1) s_{m-1} + \gamma (2^{m+1}-2) s_{m-2}  \ge 0.
$$
This is rewritten as
$$
2^{m} s_{m-1} + \gamma (2^{m+1} - 1)  s_{m-2} - s_{m} \ge 0.
$$
We reverse this inequality by multiplying it by $-\gamma \le 0$, and add \eq{ind10} to it to obtain
$$
-\gamma(2^m -1)s_{m-1} + s_{m+1}
-\gamma 2^m s_{m-1} - \gamma (2^{m+1}-1) \gamma s_{m-2} + \gamma s_{m} \le 1,
$$
which reduces to
$$
-\gamma(2^{m+1} -1) s_m +s_{m+2}\le 1\, .
$$
This concludes the proof of the second induction step.
\end{proof}

\begin{lemma}
The sequence
$$
\frac{2^m s_{m+1} (\gamma) - 1}{2^m -1}, \qquad \qquad m \ge 1,
$$
is non-increasing in $m$ for $\gamma \in [-1/4,0]$.
\label{decrease0}
\end{lemma}

\begin{proof}
We prove that for $m \ge 1$,
\begin{equation*}
\frac{2^m s_{m+1} - 1}{2^m - 1} \ge
\frac{2^{m+1} s_{m+2} - 1}{2^{m+1} - 1} \, .
\label{mono}
\end{equation*}
Using the recurrence relation~\eqref{rec:s}, this is equivalent to
$$
s_{m+1} \ge \gamma (2^{m+1} - 2) s_{m} + 1 ~~\iff~~s_{m+2} - \gamma (2^{m+1}-1) s_{m} \ge 1,
$$
which follows directly from  Lemma~\ref{gamlemma}.
\end{proof}

\begin{lemma}
The sequence
$$
\frac{2^m s_{m+1} (\gamma) - 1}{m(m+1)}, \qquad \qquad m \ge 1,
$$
is nondecreasing in $m$ for $\gamma \ge -1/4$.
\label{decrease4}
\end{lemma}

\begin{proof}
We use induction to show that for $m \ge 1$
\begin{equation*}
\frac{2^m s_{m+1} - 1}{m(m+1)} \le
\frac{2^{m+1} s_{m+2} - 1}{(m+1)(m+2)},
\label{mono20}
\end{equation*}
or equivalently, for $m\ge 1$,
\begin{equation}
(m+2)2^{m} s_{m+1} \le m 2^{m+1} s_{m+2} + 2\, .
\label{indas0}
\end{equation}

The inequality \eq{indas0} holds for $m= 1$
 as $6s_2 = 6(1+2\gamma) = 4 (1+3\gamma) + 2 = 4s_3+2$.
Using \eq{rec:s} to expand $s_{m+2}$ in \eq{indas0} we obtain
$$
(m+2)2^{m} s_{m+1} \le  m 2^{m+1} (s_{m+1} + \gamma s_{m}) + 2,
$$
and \eq{indas0} is equivalent to
\begin{equation*}
2^m s_{m+1} - \gamma m 2^{m+1} s_m \le (m-1) 2^m s_{m+1} + 2\, .
\label{eqform0}
\end{equation*}
It suffices to prove that
\begin{equation}
2^m s_{m+1} - \gamma m 2^{m+1} s_m \le (m+1)2^{m-1} s_m \, ,
\label{aux55}
\end{equation}
since the induction assumption \eq{indas0}  for $m \rightarrow m-1$ implies
$(m+1)2^{m-1} s_m \le (m-1) 2^m s_{m+1} + 2$.
But \eq{aux55}  follows directly from \eq{TP1} of Lemma~\ref{doubling} as it  is equivalent to
$2s_{m+1} \le \left[1 + m(1+4\gamma)\right] s_m$.
\end{proof}

Finally, we prove two lemmas that provide bounds for growth of the sequence $\left\{s_{m}(\gamma)-1\right\}$.

\begin{lemma}
The sequence
$$
(s_{m} (\gamma) - 1)/m, \qquad \qquad m \ge 3,
$$
is non-decreasing in $m$ for $\gamma \in [-1/4,0)$.
\label{decrease2}
\end{lemma}

\begin{proof}
The statement is equivalent to $(m+1) s_{m} \le m s_{m+1} +1$, which we prove by induction.
First, for $m = 3$ we have
$4s_3 < 3s_4 + 1$, i.e.,
$4(1+3\gamma) < 3 (1+4\gamma + 2 \gamma^2)+1$ which holds for $\gamma \neq 0$.

Assume that the statement holds for  $m \rightarrow m-1$, i.e., $ms_{m-1} \le (m-1) s_{m} +1$, which is equivalent to $s_m \le m(s_m - s_{m-1})+1$. Thus $s_m \le m\gamma s_{m-2} + 1$. However, for $\gamma \in [-1/4,0)$ and $m \ge 2$ one has  $0 < s_{m-1} < s_{m-2}$ and thus $s_m \le m \gamma s_{m-1} + 1$. The claim follows by an application of \eq{rec:s} to $s_{m-1}$.
\end{proof}

\begin{lemma}
The sequence
$$
(s_{m+1} (\gamma) - 1)/(2^{-m}-1), \qquad \qquad m \ge 1,
$$
is (i) non-decreasing in $m$,  for $\gamma \in [-1/4,0)$; (ii)
non-increasing in $m$, for $\gamma > 0$.
\label{increase2}
\end{lemma}

\begin{proof}
First, we prove (i), which is equivalent to $(2^{m+1}-2)s_{m+2} +1 \le (2^{m+1} - 1)s_{m+1}$.
Using \eq{rec:s} in the form $s_{m+2} = s_{m+1} + \gamma s_m$, this reduces to
$s_{m+1} - 2\gamma (2^m - 1)s_m \ge 1$. This follows directly from a combination of
$-\gamma (2^m -1) s_{m-1} + s_{m+1} \ge 1$, which holds for all $m \ge 2$, and $\gamma \in [-1/4,0)$ by Lemma~\ref{gamlemma} and
$s_{m-1} \le 2 s_m$ (see \eq{db1}).

Next we prove (ii) by an analogous argument. We have to show that
$(2^{m+1}-2)s_{m+2} +1 \ge (2^{m+1} - 1)s_{m+1}$,
which reduces (by \eq{rec:s} in the form $s_{m+2} = s_{m+1} + \gamma s_m$) to
$s_{m+1} - 2\gamma (2^m - 1)s_m \le 1$. This follows from
$-\gamma (2^m -1) s_{m-1} + s_{m+1} \le 1$ (by Lemma~\ref{gamlemma}) and $s_{m-1} \le 2 s_m$ (by \eq{db1}) for all $m \ge 2$.
\end{proof}

\section*{Acknowledgment}
This work was supported by the Slovak Research and Development Agency under the contract No.~APVV-14-0378, by the Scientific Grant Agency of the Slovak Republic under the grant 1/0755/19 (RK) and by the National Science Foundation under grant number NSF-DMS-1522677 (BD). Any opinions, findings, and conclusions or recommendations expressed in this material are those of the authors and do not necessarily reflect the views of the funding sources.
The authors wish to thank Casa Mathem{\'a}tica Oaxaca and Erwin Schr{\H o}dinger Institute for their hospitality during the development of the ideas for this work. To appear in SIAM Journal on Mathematical Analysis. Published on arXiv with permission of The Society for Industrial and Apllied Mathematics (SIAM).

\end{document}